\documentclass[draft, a4paper, 11pt]{article}

\usepackage{amsmath}
\usepackage{amsthm}
\usepackage{amstext}
\usepackage{amsfonts}
\usepackage{amssymb}

\usepackage{enumerate}
\usepackage{a4wide}

\def\R{\mathbb{R}}

\def\Cc{\mathcal{C}}

\def\Fc{\mathcal{F}}

\def\Oc{\mathcal{O}}
\def\Ac{\mathcal{A}}

\def\Adm{\mathrm{Adm}}
\def\sgn{\mathrm{sgn}}
\def\BV{\mathrm{BV}}
\def\Lip{\mathrm{Lip}}
\def\bu{\bar{u}}
\def\bt{\bar{t}}
\def\bx{\bar{x}}

\def\bc{\bar{c}_{\epsilon,m}}
\def\tc{\tilde{c}_{\epsilon,m}}
\def\bd{\bar{d}_{\epsilon,m}}
\def\td{\tilde{d}_{\epsilon,m}}

\newtheorem{theo}{Theorem}
\newtheorem{prop}{Proposition}
\newtheorem{lemm}{Lemma}
\newtheorem{defi}{Definition}
\newtheorem{rema}{Remark}

\title{Asymptotic stabilization of stationnary shock waves using a boundary feedback law.}

\author{Vincent Perrollaz\footnote{Institut Denis Poisson, Universit\'e de Tours, CNRS UMR
    7013, Universit\'e d'Orl\'eans}\footnote{vincent.perrollaz@lmpt.univ-tours.fr} } 

\begin{document}

\maketitle

\begin{abstract}
    In this paper we consider scalar conservation laws with a convex flux. Given a stationnary
    shock, we provide a feedback law acting at one boundary point such that this solution is now asymptotically
    stable in $L^1$-norm in the class of entropy solution. 
\end{abstract}
\tableofcontents

\newpage
\section{Generalities and previous results.}

Scalar conservation laws in one dimension are equations of the form
\begin{equation}\label{eq:1}
    u_t+(f(u))_x=0,
\end{equation}
where $u:\R\to\R$ and $f:\R\to\R$.

They are used, for instance, to model traffic flow or gas networks, but their importance  also
lies in being a first step in the  understanding of systems  of conservation  laws
$u:\R\to\R^d$. Those  systems of
equations model a  huge number  of physical  phenomena: gas
dynamics,   electromagnetism,  magneto-hydrodynamics,   shallow  water
theory, combustion theory\ldots see \cite[Chapter2]{Dafermos2016}.

For equations such as \eqref{eq:1}, the Cauchy problem on the whole line is well posed  in small
time in the framework of classical  solutions and with a $\Cc^1$ initial value. However those solutions generally blow up  in
finite time:
shock  waves   appear.  Hence to get global in time results,  a  weaker   notion  of
solution is called for.

In \cite{Oleinik1956} Oleinik proved that given a flux $f\in\Cc^2$ such that $f''>0$ and any $u_0\in L^\infty(\R)$ there exists a unique
weak solution to:
\begin{gather}
    u_t+(f(u))_x=0,\quad \quad x\in\R\quad \text{and}\quad t>0,\\
    u(0,.)=u_0,
\end{gather}
satisfying the additional condition:
\begin{equation}
    \frac{u(t,x+a)-u(t,x)}{a}\leq \dfrac{E}{t} \quad \quad \text{for}\quad x\in \R,\quad t>0,\quad
    \text{and}\quad a>0. \label{eq:2}
\end{equation} 
Here $E$ depends only on the quantities $\inf(f'')$ and $\sup(f')$ taken on $[-||u_0||_{L^\infty},||u_0||_{L^\infty}]$ and not on
$u_0$.

Later in \cite{Kruzkov1970}, Kruzkov extended this global result to the multidimensional problem, with a $C^1$  flux $f:\R \rightarrow \R^n$
not
necessarily convex:
\begin{equation}
    u_t+\mathrm{div}(f(t,x,u))=g(t,x,u),\quad \quad\text{for }t>0\quad\text{and } x\in \R^n. 
\end{equation} 
This time the weak entropy solution is defined as satisfying the following integral
inequality:\newline
    for all real numbers $k$ and all non-negative functions $\phi\in\Cc^1(\R^2)$
    \begin{multline}
    \iint_{\R^2}|u-k|\phi_t+\mathrm{sgn}(u-k)(f(u)-f(k))\nabla\phi+\mathrm{sgn}(u-k)g(t,x,u)\phi
    dtdx\\
    +\int_{\R}{u_0(x)\phi(0,x)dx}\geq 0.
\end{multline}

The initial boundary value problem for equation \eqref{eq:1} is also well posed as shown by
Leroux in \cite{leRoux1977} for the one
dimensional case with $\BV$ data, by
Bardos, Leroux and N\'ed\'elec in \cite{BLN1979} for the multidimensional case with $C^2$ data
and later by Otto in \cite{Otto1996} (see also
\cite{MNR1996}) for $L^\infty$ data. However the meaning of the boundary condition is quite intricate and the Dirichlet condition may not
be fulfilled pointwise a.e.\ in time. We will go into further details later.

Before describing in detail our particular problem, let us recall a few general facts on
general control systems. Consider such a system :
\begin{equation}
    \begin{cases}
        \dot{X}=F(X,U) , \\
        X(t_0)=X_0, 
    \end{cases}
    \label{eq:3}
\end{equation}
($X$ being the state of the system belongs to the space $\mathcal{X}$ and $U$ the so called control belongs to the space
$\mathcal{U}$), we can consider two classical problems (among others) in control theory.
\begin{enumerate}
    \item First the exact controllability problem which consists, given two states $X_0$ and $X_1$ in $\mathcal{X}$ and a positive time
        $T$, in finding a
        certain function $t\in[0,T] \mapsto U(t)\in \mathcal{U}$ such that the solution to~\eqref{eq:3}  satisfies $X(T)=X_1$.
    \item If $F(0,0)=0$, the problem of asymptotic stabilization by a stationary feedback law asks to find a function of the state
        $X\in \mathcal{X} \mapsto U(X)\in \mathcal{U}$, such that for any state $X_0$ a maximal solution $X(t)$ of the closed loop system:
        \begin{equation}
            \begin{cases} 
                \dot{X}(t)=F(X(t),U(X(t))),\\
                X(t_0)=X_0,
            \end{cases} 
        \end{equation}
        is global in time and satisfies additionally:
        \begin{gather}
            \forall R>0,\ \exists r>0 \quad\text{such that } ||X_0||\leq r \quad \Rightarrow \quad \forall t \in \R,\ ||X(t)||\leq R,\\ 
            X(t)\underset{t \to +\infty}{\to}0.
        \end{gather}
\end{enumerate}
The asymptotic stabilization property might seem weaker than exact controllability : for any
initial state $X_0$, we can find $T$ and $U(t)$ such that the solution to~\eqref{eq:3}
satisfies $X(T)=0$ in this way we stabilize $0$ in finite time. However this method suffers
from a lack of robustness with respect  to perturbation: with any error on the model, or on the
initial state, the control may not act properly anymore since at most we reach a close
neighbourhood of the state $0$. But if that stationnary state is unstable we then deviate
significantly.  This motivates the problem of asymptotic stabilization by a stationary feedback law which is 
more robust. Indeed in the case of perturbations,  once we deviate enough from $0$, the control
acts up again and drive us toward $0$. An additionnal property garanteeing a good robustness
with respect to perturbations is the existence of a Lyapunov functionnal. In finite dimension
it is often the case that if we can find a feedback function $\mathcal{U}$ stabilizing the
stationnary state, we can find another one for which we additionnally have a Lyapunov function. 

We are interested in the controllability properties of \eqref{eq:1} when we use the boundary
data as controls. In the framework of entropy solutions, some results exist for the exact
controllability problem
problem, see \cite{ADGR2015}, \cite{ADM2017}, \cite{AGV2014} \cite{AM1998}, \cite{AM2007},
\cite{AC2005}, \cite{BC2002}, \cite{Glass2007}, \cite{Glass2014}, \cite{GG2007},
\cite{Horsin1998}, \cite{Leautaud2012}, \cite{LY2017}, \cite{Perrollaz2012}. See also \cite{GZ2017} for a related problem.

Once we look at the problem of asymptotic stabilization in a classical framework the
litterature is huge see the book \cite{BC2016} for an up to date bibliography.  In the
framework of entropy solution however the only existing articles (known to the author) are
\cite{CEGGP2017}, \cite{BLDPB2017} and \cite{Perrollaz2013}. Furthermore in those articles the
goal is to stabilize a stationnary state which is actually regular. The entropy framework is
only used to garantee more stability. In this paper we aim to stabilize results particular to
the entropy framework : stationnary shock waves.

\section{Discussion on the problem and on the proofs}

Let us present rather informally and in the simpler case of Burgers' equation the problems we are interested in, the
kind of result we want to obtain and the idea behind the proofs.

Burgers' equation is the simplest equation of type \eqref{eq:1}, it reads
\begin{equation} \label{eq:5}
    \partial_t u+\partial_x\left(\frac{u^2}{2}\right)=0. 
\end{equation}
If we look at the regular stationnary states it is clear that we have the constant states
indexed by $\R$. For any real number $k\in \R$ the function $u_k$ defined by 
\begin{equation}\label{eq:6}
\forall x\in \R,\qquad u_k(x):=k,
\end{equation}
is obviously a solutions of \eqref{eq:5}. 

If we consider a stationnary entropy solution $u$, it is clear, since it is a weak solution that $u^2$ is a constant. Futhermore
using Oleinik's estimate \eqref{eq:2} which is valid for any time $t$ since $u$ is stationnary
that $u$ is actually decreasing. In the end we see that the family of such solutions is
described by a positive real number $k$ and a real number $p$ through
\begin{equation}\label{eq:7}
\forall x\in \R,\qquad u_{k,p}(x):=
\begin{cases}
    k & \text{if } x<p,\\
    -k & \text{if }x>p.
\end{cases}
\end{equation}
Of course we see that at the discontinuity $x=p$, the Rankine-Hugoniot condition holds
$$0=\frac{\frac{k^2}{2}-\frac{(-k)^2}{2}}{k-(-k)}=\frac{f(k)-f(-k)}{k-(-k)}.$$

Now let us look at the stability of those stationnary solutions.
\begin{itemize}
    \item For the family $u_k$ defined by \eqref{eq:6}, using the results of
        \cite{Dafermos2016} Chapter 11 section 8, if we consider an initial data $u_0$ and a number $\epsilon>0$
        such that 
        $$\forall x\in \R,\qquad |u_k(0,x)-u_0(x)|\leq \epsilon,$$
        then we have for the solution $u$ of \eqref{eq:5} corresponding to $u_0$
        $$\forall x \in \R,\qquad |u_k(t,x)-u_(t,x)|\leq \epsilon.$$
        So we have stability (though not asymptotic stabilization) of $u_k$ in the $L^\infty$
        setting. 
    \item Let us now consider \eqref{eq:5} on the interval $(0,L)$ with additionnal boundary
        conditions
        \begin{equation} \label{eq:8}
            \begin{cases}
                u(t,0)=\alpha, \\
                u(t,L)=\beta,
            \end{cases}
        \end{equation}
        once again let us mention that we cannot expect those boundary conditions to hold for
        a.e.\ time $t$. This is related to the presence of boundary layers at the borders, we
        will make a precise statement on the sense of the boundary conditions in the next part.

        It can be shown using generalized characteristics (see \cite{Perrollaz2013}) that if
        $k\neq 0$, $\alpha=\beta=k$ there exists a time $T$ such that  for $u_0\in
        L^\infty(0,L)$ then the entropy solution $u$ satisfy
        $$\forall t\geq T,\qquad \forall x\in (0,L),\qquad u(t,x)=k.$$
        This is enough to show the asymptotic stabilization in $L^\infty(0,L)$ (and of course
        also in $L^1$) toward $u_k$.

        As for robustness result, let us suppose that $\alpha,\beta>0$ then we have a time
        $T>0$ such that for any initial data $u_0$ the entropy solution $u$ satisfies
        $$\forall t\geq T,\qquad \forall x\in (0,L),\qquad u(t,x)=\alpha,$$
        so as long as $\alpha$ is close to $k$ we still have some reasonnable asymptotics.

    \item On the other hand for $k>0$ if we look at the family $(u_{k,p})_{p\in (0,L)}$ it is clear that 
        all  those solutions satisfy \eqref{eq:8} with $\alpha=k$ and $\beta=-k$. Since 
        $$||u_{k,p}-u_{k,p'}||_{L^1(0,L)}=2|p-p'|\cdot k,$$
        we already see that we cannot expect asymptotic stabilization for this family in
        $L^1$. In the $L^\infty$ setting we have the following result from \cite{MT1999}, if
        $\alpha=k$ and $\beta=-k$ there exists a time $T>0$ such that for any initial data
        $u_0\in L^\infty$ there exists $p\in (0,L)$ such that the entropy solution to
        \eqref{eq:5},\eqref{eq:8} satisfies
        $$\forall t\geq T,\qquad \forall x\in (0,L),\qquad u(t,x)=u_{k,p}(x),$$
        but the position $p$ of the singularity does depend on $u_0$, so we basically cannot
        expect asymptotic stabilization in $L^\infty$, though simple stability may still hold.

        As far as robustness is concerned, it can be shown (using the results on generalized
        characteristics of \cite{Perrollaz2013}) for instance that if $\alpha>k$ and
        $\beta=-k$ then we have a time $T>0$ such that for any initial data $u_0$ the entropy
        solution $u$ satisfies
        $$\forall t\geq T,\qquad \forall x\in (0,L),\qquad u(t,x)=\beta,$$
        and even starting from $u_{k,p}$ we go far from it in $L^\infty$ and in $L^1$.
\end{itemize}
For a more precise discussion of the above see \cite{MT1999}.

Following the previous results, the goal is now, given a stationnary state $u_{k,p}$ to provide
a feedback law for the boundary conditions such that $u_{k,p}$ is asymptotically stable for the
semigroup.

To that end the idea is (very roughly) the following. According to the results of \cite{MT1999} we
can expect that if we inject $\alpha=k$ and $\beta=-k$ in the system after some time we get a stationnary
shock wave $u_{k,p'}$, now we want to move the singularity from $p'$ to $p$, to that end we
oserve the value of $u(t,.)$ at $p$ if it is $k$ then $p'<p$ and so the singularity needs to
move to the right, so we modify $\alpha$ to be a bit more than $k$, after some time the trace
to the left of the singularity will be this state so using the rankine Hugoniot condition the 
singularity will move with positive speed.  Of course with $p'>p$ we set $\alpha$ a bit less
than $k$ so after some time the singularity will move to the left. 

In practice there are multiple difficulties when we want to implement the above strategy.
\begin{enumerate}
    \item Since we are in feedback form with no access to $t$, we cannot wait for the profile
        to be a $u_{k,p'}$ before using the second strategy which basically reduces the dynamic
        to a 1d phenomenon.
    \item The time it takes for the inbound $\alpha$ to get to the singularity depends on the
        position of the singularity and of the state $\alpha$. So basically we expect than rather
        than some scalar ODE on the position of the singularity we end up with a delayed
        differential equation with a delay depending on the solution itself.
    \item We will get some kind of oscillatory phenomenon of the singularity around the goal
        $p$, we need to make sure that there is some kind of "damping".
    \item The regularity will be $L^\infty_t\BV_x\cap \Lip_tL^1_x$ so we need some kind of filtered value of $u(t,.)$ near
        $p$.
\end{enumerate}
Let us discuss now the content of the following sections. In Section~\ref{sec:MainResult} we
will provide the main result and some definitions necessary for it. In
Section~\ref{sec:EntropySolutionandBoundaryconditions} we provide the remaining definitions
necessary for the proof. In~\ref{sec:ExistenceandUniqueness}, we will show that the closed loop system does
have a unique solution which depends continuously of the initial data. In
Section~\ref{sec:GeneralizedCharacteristicsAndTheBoundary} we will provide results on generalized
characteristics in particular their interactions with the boundary. They will be our main tool
to study the solutions. In Section~\ref{sec:AsymptoticStabilization} we will prove the main
result using a Lemma on delayed differential equations which itself is proved
in~\ref{sec:DelayedDifferentialEquation}.

\section{Main Result}
\label{sec:MainResult}
\begin{defi}\label{def:hypoFlux}
    In the whole paper we will suppose the following fixed.
    \begin{itemize}
        \item The flux $f:\R\to \R$ will be a $\Cc^2$ uniformly convex function, so in
            particular
            $$\underset{u\to \pm \infty}{\lim} f(u)=+\infty.$$
            We will additionnally suppose that 
            $$\min f = f(0)=0,$$
            but this is not restrictive since given $a$ and $b$ the flux change
            $$\tilde{f}(u):=f(a+u)-b,$$
            sends entropy solution on entropy solution.
        \item Given a positive number $m$ we can now define the numbers $u_l(m)$ and $u_r(m)$
            satisfying
            $$u_l(m)<u_r(m),\qquad f(u_l(m))=f(u_r(m))=m.$$
        \item We can now define another family of stationnary solutions. Let us consider $m>0$
            and $\alpha \in (0,L)$ we define
            \begin{equation}\label{eq:4}
                \forall (t,x)\in \R\times (0,L),\qquad \bu_{\alpha,m}(t,x):=
                \begin{cases}
                    u_l(m) & \quad\text{if } x<\alpha,\\
                    u_r(m) & \quad\text{if } x\geq \alpha.
                \end{cases}
            \end{equation}
    \end{itemize}
\end{defi}
\begin{proof}
    Since $\frac{f(u_l(m))-f(u_r(m))}{u_l(m)-u_r(m)}=0$ the Rankine-Hugoniot
    condition is satisfied and thus $\bu_{\alpha,m}$ is indeed a weak solution.

    Since $u_l(m)<u_r(m)$ and $f$ is convex the following entropy
    condition is also satisfied. For any $k\in (u_l(m), u_r(m))$,
    $$ \frac{f(u_l(m))-f(k)}{u_l(m)-k}\geq
    \frac{f(u_l(m))-f(u_r(m))}{u_l(m)-u_r(m)}\geq
    \frac{f(k)-f(u_r(m))}{k-u_r(m)}.$$
\end{proof}
To describe the feedback law we will need the following functions. We suppose that we are given an
interval $[0,L]$ and a position $\alpha \in (0,L)$.
\begin{defi}
    Let us consider three
    positive numbers $\epsilon,\ \delta, \nu$. (Those will be parameters to be tuned later on) We will suppose that
    $[\alpha-\delta,\alpha+\delta]\subset (0,L)$ and define the functions.
    \begin{equation}\label{eq:defAc}
        \forall z\in \R,\qquad \Ac_{\epsilon, \nu}(z):=
        \begin{cases}
            -\epsilon & \quad\text{if } z\leq -\nu,\\
            \epsilon\frac{z}{\nu} & \quad\text{if } -\nu\leq z \leq \nu\\
            \epsilon & \quad\text{if } \nu \leq z
        \end{cases}.
    \end{equation}
    \begin{equation}\label{eq:defOc}
        \forall u\in L^1(0,L),\qquad
        \Oc_{\alpha,\delta}(u):=\frac{1}{2\delta}\int_{\alpha-\delta}^{\alpha+\delta}{(u(x)-\bu_{\alpha,m})dx}.
    \end{equation}
\end{defi}

We will now be interested in the solutions the following closed loop system
\begin{equation}\label{eq:SystemeFerme}
    \begin{cases}
        \partial_t u +\partial_x f(u)=0,\\
        u(t,0)"="u_l(m)-\Ac_{\epsilon,\nu}(\Oc_{\alpha,\delta}(u(t,.))),\\
        u(t,L)"="u_r(m),\\
        u(0,x)=u_0(x)
    \end{cases}
\end{equation}
\begin{theo}\label{thm:principal}
    Given $L$, $\alpha$, $m$ and $\delta$ we can find $\epsilon$ and $\nu$ small enough such
    that given $u_0\in \BV(0,L)$ the system \eqref{eq:SystemeFerme} has a unique entropy
    solution $u$. Furthermore there are constants $C,M>0$ such that
    \begin{equation}\label{eq:StabiliteAs}
        \forall t\geq 0,\qquad ||u(t,.)-\bu_{\alpha,m}||_{L^1(0,L)}\leq
        Me^{-Ct}||u_0-\bu_{\alpha,m}||_{L^1(0,L)}.
    \end{equation}
\end{theo}
\begin{rema}
    \begin{itemize}
        \item In the proofs we will precise the way $\epsilon$ and $\nu$ must be chosen.
        \item     Note that we have chosen to act at the left boundary but the same result would hold with an
            action at the right boundary. 
        \item The convexity of $f$ is  however crucial to the analysis.
    \end{itemize}
\end{rema}

\section{Entropy solution and Boundary conditions}
\label{sec:EntropySolutionandBoundaryconditions}

We need to precise the sense in which we consider the solutions since we have both regularity
problems and overdetermined boundary conditions (see \cite{Dafermos2016} for a general exposition). We
will follow \cite{leRoux1976} and \cite{BLN1979}. (one can also look at \cite{AWC2006} and
\cite{CR2015} for more
general and up to date results)
We will need the notations
$$\forall (a,b)\in \R^2,\qquad I(a,b):=[\min(a,b),\max(a,b)].$$
$$\forall z\in \R,\qquad \sgn(z):=
\begin{cases}
    1 & \quad\text{if } z>0\\
    -1 & \quad\text{if }z<0\\
    0 & \quad\text{otherwise}
\end{cases}$$
\begin{defi}
    We say that a function $u\in L^\infty([0,+\infty ); \BV(0,L))$ is an entropy solution of
    \eqref{eq:SystemeFerme} when for any number $k\in \R$ and any positive function $\phi\in
    \Cc^1_c(\R^2)$ we have
    \begin{multline}\label{eq:SolutionFaible}
        \int_0^{+\infty} \int_0^L{|u(t,x)-k|\partial_t
        \phi(t,x)+\sgn(u(t,x)-k)(f(u(t,x))-f(k))\partial_x \phi(t,x)dx dt} \\
        \int_0^{+\infty}{\sgn(u_r(m)-k)(f(k)-f(u(t,L^-)))\phi(t,L)dt}\\
        -\int_0^{+\infty}{\sgn(u_l(m)-\Ac_{\epsilon,\nu}(\Oc_{\alpha,\delta}(u(t,.)))-k)(f(k)-f(u(t,0^+)))\phi(t,0)dt}\\
        +\int_0^L{|u_0(x)-k|\phi(0,x)dx}\geq 0
    \end{multline}
\end{defi}
Let us be more explicit on the sense in which the boundary conditions hold.
\begin{defi}\label{def:defConditionsBords}
    For $u\in \R$ we define $\Adm_l(u)$ and $\Adm_r(u)$to be
    $$
    \Adm_l(u):=
    \begin{cases}
        \{z\in \R\ :\ f'(z)\leq 0\} & \quad\text{if } f'(u)\leq 0\\
        \{z\in \R\ :\ f'(z)<0 \quad\text{and }f(z)\geq f(u)\} \cup \{u\} & \quad\text{if } f'(u)>0
    \end{cases}
    $$
    $$
    \Adm_r(u):=
    \begin{cases}
        \{z\in \R\ :\ f'(z)\geq 0\} & \quad\text{if } f'(u)\geq 0\\
        \{z\in \R\ :\ f'(z)>0 \quad\text{and }f(z)\geq f(u)\} \cup \{u\} & \quad\text{if } f'(u)<0
    \end{cases}
    $$
\end{defi}
At the right boundary we ask that for almost all time $t\geq 0$
$$u(t,L^-)\in \Adm_r(u_r(t)),$$
which means
\begin{equation}\label{eq:ConditionBordDroit}
    \forall k\in I(u(t,L^-), u_r(m)),\qquad \sgn(u(t,L^-)-u_r(m))(f(u(t,L^-))-f(k)) \geq
    0,
\end{equation}

At the left boundary we ask that for almost all time $t\geq 0$
$$u(t,0^+)\in \Adm_l(u_l(t)),$$
\begin{multline}\label{eq:ConditionBordGauche}
    \forall k\in I\big(u(t,0^+),
    u_l(m)-\Ac_{\epsilon,\nu}(\Oc_{\alpha,\delta}(u(t,.)))\big),\\
    \sgn\Big(u(t,0^+)-(u_l(m)-\Ac_{\epsilon,\nu}(\Oc_{\alpha,\delta}(u(t,.))))\Big)(f(u(t,0^+))-f(k)) \leq
    0,
\end{multline}
\begin{rema}
    The formulation in term of admissibility set depends on the convexity of $f$ while
    \eqref{eq:ConditionBordDroit} and \eqref{eq:ConditionBordGauche} are more general.
\end{rema}
\section{Existence and Uniqueness}
\label{sec:ExistenceandUniqueness}
In this part we consider a fixed $u_0\in \BV(0,L)$ and we want to show the existence and
uniqueness of a solution to the closed loop system \eqref{eq:SystemeFerme}.

Let us first recall the following result from \cite{leRoux1976}, \cite{BLN1979}.
\begin{theo}\label{thm:resultatBoucleOuverte}
    Given any time $T>0$ and functions $u_0\in \BV(0,L)$, $v_l\in \BV_{loc}(0,+\infty)$ and $v_r\in
    \BV_{loc}(0,+\infty)$ there exists a unique entropy solution $v\in
    L^\infty_{loc}((0,+\infty);\BV(0,L))\cap
    \Lip_{loc}(\R^+;L^1(0,L))$ to 
    \begin{equation}\label{eq:SystemOuvert}
        \begin{cases}
            \partial_t v+\partial_x f(v)=0,\\
            v(t,0)=v_l(t),\\
            v(t,L)=v_r(t),\\
            v(0,x)=v_0(x).
        \end{cases}
    \end{equation}
\end{theo}
Once again we interpret a solution of \eqref{eq:SystemOuvert} to mean 
\begin{multline*}
    \int_0^{+\infty} \int_0^L{|v(t,x)-k|\partial_t
    \phi(t,x)+\sgn(v(t,x)-k)(f(v(t,x))-f(k))\partial_x \phi(t,x)dx dt} \\
    \int_0^{+\infty}{\sgn(v_r(t)-k)(f(k)-f(v(t,L^-)))\phi(t,L)dt}\\
    -\int_0^{+\infty}{\sgn(v_l(t)-k)(f(k)-f(v(t,0^+)))\phi(t,0)dt}\\
    +\int_0^L{|v_0(x)-k|\phi(0,x)dx}\geq 0,
\end{multline*}
for any number $k$ and any positive function $\phi\in \Cc^1(\R^2)$.
\begin{defi}
    Given a function $z\in L^\infty(\R^+)\cap \Lip(\R^+)$ we use the previous result to get
    $u\in L^\infty_{loc}(\R^+;\BV(0,L))\cap \Lip_{loc}(\R^+; L^1(0,L))$ the solution to
    \begin{equation}
        \begin{cases}
            \partial_t u+\partial_x f(u)=0,\\
            u(t,0)=u_l(m)-\Ac_{\epsilon,\nu}(z(t)),\\
            u(t,L)=u_r(m),\\
            u(0,x)=u_0(x).
        \end{cases}
    \end{equation}
    We will now define the operator $\mathcal{F}$ by 
    \begin{equation}\label{eq:DefinitionOperateur}
        \forall t\geq 0,\qquad \Fc(z)(t):= \Oc_{\alpha,\delta}(u(t,.)).
    \end{equation}
\end{defi}
We now recall another result from \cite{MT1999}.
\begin{prop}\label{prop:Comparaison}
    If we consider initial data $v_0$, $w_0$ in $\BV(0,L)$ and boundary data $v_l$, $v_r$,
    $w_l$ and $w_r$ in $\Lip([0,T])$, the solutions $v$ and $w$ of 
    \begin{equation}
        \begin{cases}
            \partial_t v+\partial_x f(v)=0,\\
            v(t,0)=v_l(t),\\
            v(t,L)=v_r(t),\\
            v(0,x)=v_0(x),
        \end{cases}
        \qquad 
        \begin{cases}
            \partial_t w+\partial_x f(w)=0,\\
            w(t,0)=w_l(t),\\
            w(t,L)=w_r(t),\\
            w(0,x)=w_0(x),
        \end{cases}
    \end{equation}
    satisfy
    \begin{multline}\label{eq:Comparaison}
        \forall T>0,\qquad \int_0^L{(v(T,x)-w(T,x))^+dx}\leq
        \int_0^L{(v_0(x)-w_0(x))^+dx}\\
        +\int_0^T{(v_l(t)-w_l(t))^++(v_r(t)-w_r(t))^+dt}.
    \end{multline}
    Where we used 
    $$\forall r\in \R,\qquad r^+:=\max(0,r).$$
\end{prop}
\begin{proof}
    This is just a particular case of Theorem 2.4 in \cite{MT1999}.
\end{proof}
\begin{prop}
    The space $L^\infty(\R^+)\cap \Lip(\R^+)$ is stable under $\Fc$. Furthermore $\Fc$ has
    a unique fixed point on this space.
\end{prop}
\begin{proof}
    \begin{itemize}
        \item We note that using the definition of $\Ac_{\epsilon,\nu}$ we get
            $$\forall t\geq0,\qquad \Ac_{\epsilon,\nu}(z(t))\in [-\epsilon,\epsilon],$$
            therefore with 
            $$C:=\max\big(||u_0||_{L^\infty(0,L)}, |u_l(m)|+\epsilon,
            |u_r(m)|\big),$$
            we see that the constant function $C$ (resp. $-C$) is solution of the system which is
            greater (resp.\ smaller) than $u$ on the boundary so using Proposition~\ref{prop:Comparaison} we see
            that we have
            $$\forall (t,x)\in \R^+\times [0,L],\qquad -C\leq u(t,x)\leq C.$$
        \item For the next part of the result we use $k=\pm C$ in the definition of an entropy solution
            with a test function $\phi$ which has a support in $(0,+\infty)\times (0,L)$ to get
            \begin{equation}\label{eq:egalite}
                \int_0^{+\infty}\int_0^L{u(t,x)\partial_t \phi(t,x)+f(u(t,x))\partial_x \phi(t,x)dxdt}=0.
            \end{equation}
            A classical density argument shows that the equation above is still admissible if $\phi$
            is just Lipschitz.

            Now given a time $T$ positive numbers $h$ and  $\theta$ we define 
            $$\phi_\theta(t,x):=\psi_\theta(t)\kappa_\theta(x),$$
            with
            $$
            \psi_\theta(t):=
            \begin{cases}
                0 & \quad\text{if }t\leq T-\theta\\
                \frac{t-T-\theta}{\theta} & \quad\text{if }T-\theta \leq t\leq T\\
                1 & \quad\text{if } T\leq t \leq T+h\\
                \frac{T+h+\theta-t}{\theta} & \quad\text{if }T+h\leq t\leq T+h+\theta\\
                0 & \quad\text{otherwise,}
            \end{cases}$$
            $$\kappa_\theta(x):=
            \begin{cases}
                0 & \quad\text{if }x\leq \alpha-\delta -\theta\\
                \frac{x-\alpha+\delta}{\theta} & \quad\text{if } \alpha-\delta-\theta\leq x\leq
                \alpha-\delta\\
                1 & \quad\text{if } \alpha-\delta \leq x\leq \alpha+\delta\\
                \frac{\alpha+\delta+\theta-x}{\theta} & \quad\text{if } \alpha+\delta\leq x\leq
                \alpha+\delta +\mu\\
                0 & \quad\text{otherwise,}
            \end{cases}.
            $$
            Taking $\theta\to 0$ in \eqref{eq:egalite} we obtain 
            $$\int_{\alpha-\delta}^{\alpha+\delta}{u(T,x)dx}-\int_{\alpha-\delta}^{\alpha+\delta}{u(T+h,x)dx}+\int_T^{T+h}{f(u(t,\alpha-\delta))-f(u(t,\alpha+\delta))dt}=0.$$
            Using the definition of $\Oc_{\alpha,\delta}$ \eqref{eq:defOc}, the $L^\infty$ bound and the convexity of $f$ we now get
            $$|\Fc(z)(T+h)-\Fc(z)(T)|\leq h \frac{\max(f(C),f(-C))}{2\delta}.$$
        \item Consider $y$ and $z$ two Lipschitz bounded functions. Let us call $u$ and $v$ the
            entropy solution involved in the definitions of $\Fc(y)$ and $\Fc(z)$. Using Proposition~\ref{prop:Comparaison}
            we get
            \begin{align*}
                |\Fc(y)(T)-\Fc(z)(T)|&\leq \frac{1}{2\delta} \int_0^L{|u(T,x)-v(T,x)|dx}\\
                                     &\leq \frac{1}{2\delta} \int_0^T{|\Ac_{\epsilon,\nu}(y(t))-\Ac_{\epsilon,\nu}(z(t))|dt}\\
                                     &\leq \frac{1}{2\delta}\int_0^T{\frac{\epsilon}{\nu}|y(t)-z(t)|dt}\\
                                     &\leq \frac{\epsilon}{2\delta \nu}T ||y-z||_{L^\infty(0,T)}.
            \end{align*}
            This is enough to show that $\Fc$ is continuous with respect to the uniform convergence on
            any compact. But $\Fc$ takes value on a set which is uniformly bounded with equilipschitz
            functions, and is therefore a compact set for this precise topology. We can apply Schauder fixed point
            Theorem. (see \cite{Rudin1973}) 
        \item Let us now consider two such fixed points $y$ and $z$. The previous calculation but  gives 
            $$\forall T\geq 0,\qquad |y(T)-z(T)|\leq \frac{\epsilon T}{2\delta
            \nu}||y-z||_{L^\infty(0,T)}.$$
            For continuous functions $t\mapsto ||.||_{L^\infty(0,t)}$ is continuous and nondecreasing so if we define
            $$T^*:=\sup\{T\geq 0\ :\ ||y-z||_{L^\infty(0,T)}=0\},$$
            we see that if $\frac{\epsilon T}{2\delta \nu}<1$ we have $T\leq T^*$, therefore
            $$T^*\geq \frac{2\delta \nu}{\epsilon}.$$
            If we suppose that $T^*<+\infty$, since $y$ is equal to $z$ on $[0,T^*]$ so are $u$ and $v$ but then applying Proposition~\ref{prop:Comparaison}
            with $u(T,.)$ as initial data we have with the same calculation as before
            $$\forall T\geq T^*,\qquad |y(T)-z(T)|\leq \frac{\epsilon (T-T^*)}{2\delta
            \nu}||y-z||_{L^\infty(T^*,T)}.$$
            but if $\frac{(T-T^*)\epsilon}{2\delta \nu}<1$ we see that $||y-z||_{L^\infty(0,T)}=0$ so
            $T\leq T^*$ which is absurd therefore $T^*=+\infty$ and the fixed point of $\Fc$ is unique.
    \end{itemize}
\end{proof}

\section{Generalized Characteristics and the boundary}
\label{sec:GeneralizedCharacteristicsAndTheBoundary}
We describe in this section a technical tool that will be used extensively in the following to
study the local properties of the solution of the closed loop system.  We begin by recalling a
few definitions and results from \cite{Dafermos1977}. We will refer in this
section to the system 
\begin{equation}
    \begin{cases}
        \partial_t u + \partial_x (f(u))=0& \qquad \quad\text{on }  (0,+\infty)\times (0,L),\\
        u(0,.)=u_0 &\qquad  \quad\text{on } (0,L),\\
        \sgn(u(t,L^-)-u_r(t))(f(u(t,L^-))-f(k))\geq 0& \qquad \forall k \in I(u_r(t),u(t,L^-)),\ dt\ a.e., \\
        \sgn(u(t,0^+)-u_l(t))(f(u(t,0^+))-f(k))\leq 0& \qquad  \forall k \in I(u_l(t),u(t,0^+)),\ dt\ a.e.,
    \end{cases} \label{eq:syscar}
\end{equation}
where \textbf{only for this section} $u_l$ and $u_r$ are two regulated functions of time thus defined on $\R^+$, $u_0\in \BV(0,L)$ and $u$ is the unique entropy solution. 
\begin{rema}\label{rem:conditionsBord}
    If the boundary condition at $x=0$ in \eqref{eq:syscar} is satisfied at time $t$ it means
    that
    \begin{itemize} 
        \item either $u(t,0^+)=u_l(t)$
        \item or for any state $k\in I(u_l(t),u(t,0^+))$ we have
            $$\frac{f(u(t,0^+)-f(k)}{u(t,0^+)-k}\leq 0,$$
            which means that any wave generated by the Riemann problem between $u_l(t)$ and
            $u(t,0^+)$ leaves the domain.
    \end{itemize}
    the same kind of interpretation holds for the boundary condition at $x=L$.
\end{rema}
Following \cite{Dafermos1977} we introduce the notion of generalized characteristic.
\begin{defi}
    \begin{itemize}
        \item If $\gamma$ is an absolutely continuous function defined on an interval $(a,b) \subset \R^+$ and with values in $(0,L)$, we say that $\gamma$ is a generalized characteristic of \eqref{eq:syscar} if:
            \begin{equation*}
                \dot{\gamma}(t)\in I(f'(u(t,\gamma(t)^-)),f'(u(t,\gamma(t)^+))) \qquad dt\ a.e..
            \end{equation*}
            This is the classical characteristic ODE taken in the weak sense of Filippov
            \cite{Filippov1960}.
        \item A generalized characteristic $\gamma$ is said to be genuine on $(a,b)$ if:
            \begin{equation*}
                u(t,\gamma(t)^+)=u(t,\gamma(t)^-)\quad dt\ a.e..
            \end{equation*}
    \end{itemize}
\end{defi}
We recall the following results from \cite{Dafermos1977}.
\begin{theo}
    \label{thm:Daf}
    \begin{itemize}
        \item  For any $(t,x)$ in $(0,+\infty)\times(0,L)$ there exists at least one generalized characteristic $\gamma$ defined on $(a,b)$
            such that $a<t<b$ and  $\gamma(t)=x$.
        \item If $\gamma$ is a generalized characteristics defined on $(a,b)$ then for almost all $t$ in $(a,b)$:
            \begin{equation*}
                \dot{\gamma}(t)=
                \begin{cases}
                    f'(u(t,\gamma(t)) & \quad\text{if }u(t,\gamma(t)^+)=u(t,\gamma(t)^-),\\
                    \frac{f(u(t,\gamma(t)^+))-f(u(t,\gamma(t)^-))}{u(t,\gamma(t)^+)-u(t,\gamma(t)^-)} & \quad\text{if } u(t,\gamma(t)^+)\neq u(t,\gamma(t)^-).
                \end{cases}
            \end{equation*}
        \item If $\gamma$ is a genuine generalized characteristics on $(a,b)$ (with
            $\gamma(a),\gamma(b)\in (0,L)$), then there exists a $C^1$ function $v$ defined on $(a,b)$ such that:
            \begin{gather}
                u(b,\gamma(b)^+) \leq v(b) \leq u(b,\gamma(b)^-), \notag\\
                u(t,\gamma(t)^+)=v(t)=u(t,\gamma(t)^-)\quad \forall t \in (a,b),\label{eq:boundentre} \\
                u(a,\gamma(a)^-)\leq v(a) \leq u(a,\gamma(a)^+). \notag
            \end{gather}
            Furthermore $(\gamma,v)$ satisfy the classical ODE equation:
            \begin{equation}
                \begin{cases}
                    \dot{\gamma}(t)=f'(v(t)), \\
                    \dot{v}(t)=0,
                \end{cases}\quad  \forall t\in(a,b). \label{eq:carode}
            \end{equation}
        \item Two genuine characteristics may intersect only at their endpoints.
        \item If $\gamma_1$ and $\gamma_2$ are two generalized characteristics defined on $(a,b)$, then we have:
            \begin{equation*}
                \forall t \in (a,b),\quad \left(\gamma_1(t)=\gamma_2(t)\Rightarrow \forall s\geq t,\ \gamma_1(s)=\gamma_2(s)\right).
            \end{equation*}
        \item For any $(t,x)$ in $\R^+ \times (0,L)$ there exist two generalized characteristics $\chi^+$ and $\chi^-$ called maximal and minimal and associated to $v^+$ and $v^-$ by \eqref{eq:carode}, such that if $\gamma$ 
            is a generalized characteristic going through $(t,x)$  then 
            \begin{gather*} 
                \forall s\leq t,\qquad \chi^-(s) \leq \gamma(s) \leq \chi^+(s),\\
                \quad\text{$\chi^+$ and $\chi^-$ are genuine on $\{s<t\}$ },\\
                v^+(t)=u(t,x^+)\qquad \quad\text{and } \qquad v^-(t)=u(t,x^-).
            \end{gather*}
    \end{itemize}
\end{theo}
Note that in the previous theorem, every property dealt only with  the interior of $\R^+ \times
[0,L]$. The following result describe the influence of the boundary conditions on the generalized characteristics.
\begin{prop}\label{prop:carBord}
    Let $u$ be the unique entropy solution of \eqref{eq:syscar} and consider $\chi$  a genuine characteristic on an interval $[a,b]$ such that
    $$\forall t\in (a,b],\qquad \chi(t)\in (0,L),$$
    then we know from the Theorem above  that there is a constant $v\in \R$ such that
    $$\forall t\in [a,b],\qquad \dot{\chi}(t)=f'(v)$$
    and 
    $$\forall t\in (a,b),\qquad u(t,\gamma(t))=v.$$
    But then we have
    \begin{equation}\label{eq:gauche}
        \chi(a)=0 \Rightarrow u_l(a^+)\leq v\leq u_l(a^-),
    \end{equation}
    \begin{equation}\label{eq:droite}
        \chi(a)=L \Rightarrow u_r(a^-)\leq v\leq u_r(a^+).
    \end{equation}
    (Where the existence of the limits is an hypothesis)
\end{prop}
The main difficulty in the proof comes from the fact that the boundary conditions in
\eqref{eq:syscar} are satisfied only for almost all times.
Before the proof let us begin with a lemma.
\begin{lemm}\label{lemm:trapeze}
    Consider $0\leq t_0<t_1$, $0\leq x_A<x_B\leq L$ and $0\leq x_C<x_D\leq L$. We introduce
    $$s_l:=\frac{x_C-x_A}{t_1-t_0},\qquad s_r=\frac{x_D-x_B}{t_1-t_0}.$$
    We then have
    \begin{multline*}
        \int_{x_A}^{x_B}{u(t_0,x)dx}-\int_{x_C}^{x_D}{u(t_1,x)dx}\\
        +\int_{t_0}^{t_1}{|f(u(t,(x_A+s_l(t-t_0))^+))-s_lu(t,(x_A+s_l(t-t_0))^+)]}\\
        -[f(u(t,(x_B+s_r(t-t_0))^-))-s_ru(t,(x_B+s_r(t-t_0))^-)]dt=0.
    \end{multline*}
    Note that by letting $x_A$ tend to $x_B$ we have the following equality wih $x_A=x_B$.
    \begin{multline*}
        -\int_{x_C}^{x_D}{u(t_1,x)dx}\\
        +\int_{t_0}^{t_1}{[f(u(t,(x_A+s_l(t-t_0))^+))-s_lu(t,(x_A+s_l(t-t_0))^+)]}\\
        -[f(u(t,(x_A+s_r(t-t_0))^+))-s_ru(t,(x_A+s_r(t-t_0))^+)]dt=0.
    \end{multline*}
\end{lemm}
\begin{proof}
    We define
    $$\chi_l(t):=x_A+(t-t_0)s_l,\qquad \chi_r(t):=x_B+s_r(t-t_0).$$
    We can see that 
    $$\forall t\in [t_0,t_1],\qquad 0\leq \chi_l(t)< \chi_r(t)\leq L.$$
    Of course we also have
    $$\chi_l(t_1)=x_C,\qquad \chi_r(t)=x_D.$$
    We will now define for $\epsilon>0$ small enough
    $$\rho_\epsilon(t):=
    \begin{cases}
        0 & \quad\text{if } t\leq t_0\\
        \frac{t-t_0}{\epsilon} &  \quad\text{if } t_0\leq t\leq t_0+\epsilon\\
        1&  \quad\text{if } t_0+\epsilon \leq t\leq t_1-\epsilon\\
        \frac{t_1-t}{\epsilon} &  \quad\text{if } t_1-\epsilon \leq t\leq t_1\\
        0 &  \quad\text{if } t_1\leq t
    \end{cases}
    $$
    it is clear that $\rho_\epsilon$ is Lipshitz continuous and has support 
    $[t_0,t_1]$. We also need
    $$\phi_\epsilon(t,x):=
    \begin{cases}
        0 & \quad\text{if } x\leq \chi_l(t)+\epsilon \\
        \frac{x-(\chi_l(t)+\epsilon)}{\epsilon}& \quad\text{if } \chi_l(t)+\epsilon \leq x\leq
        \chi_l(t)+2\epsilon\\
        1 & \quad\text{if } \chi_l(t)+2\epsilon \leq x\leq \chi_r(t)-2\epsilon\\
        \frac{\chi_r(t)-\epsilon -x}{\epsilon} & \quad\text{if } \chi_r(t)-2\epsilon \leq x\leq
        \chi_r(t)-\epsilon\\
        0 & \quad\text{if }\chi_r(t)-\epsilon \leq x
    \end{cases}
    $$
    Now one can see that the function defined by 
    $$\forall (t,x)\in \R^2,\qquad \phi_\epsilon(t,x):=\rho_\epsilon(t)\phi_\epsilon(t,x)$$
    is Lipschitz and has compact support in $(0,+\infty)\times (0,L)$ so one can use it has a test function in the weak
    formulation of the equation that is
    $$\int_{\R^2}{u(t,x)\partial_t\psi(t,x)+f(u(t,x))\partial_x\psi(t,x)dtdx}=0.$$
    But then letting $\epsilon \to 0^+$ and remembering that $u$ is Lipshitz in time with value in
    $L^1$ we get the result.
\end{proof}
\begin{proof}[Proof of Proposition~\ref{prop:carBord}]
    We will prove the two inequalities of \eqref{eq:gauche} independently, \eqref{eq:droite} is
    a simple adaptation and so left to the reader.
    \begin{itemize}
        \item Since $f'(v)>0$ the estimate is obvious if $f'(u_l(a^+))\leq 0$ so we can suppose
            $f'(u_l(a^+))>0$ but then it implies that for $\delta$ small enough
            $$\forall t\in [a,a+\delta],\qquad f'(u_l(t))>0$$
            and using the definition of $\Adm_l$ in this case we see that for almost all $t\in
            [a,a+\delta]$ we have $f(u(t,0^+))\geq f(u_l(t))$.

            For $\epsilon>0$ small enough we apply Lemma~\ref{lemm:trapeze} to $t_0=a$,
            $t_1=a+\frac{\epsilon}{f'(v)}$, $x_A=x_B=0$, $x_C=0$ and
            $x_D=\epsilon$ to obtain
            $$-\int_0^\epsilon{u(a+\frac{\epsilon}{f'(v)},x)dx}+\int_a^{a+\frac{\epsilon}{f'(v)}}{f(u(t,0^+))-[f(v)-vf'(v)]dt}=0.$$
            Using the previous inequality we get that for $\epsilon< f'(v)\delta$ 
            $$-\int_0^\epsilon{u(a+\frac{\epsilon}{f'(v)},x)dx}+\int_a^{a+\frac{\epsilon}{f'(v)}}{f(u_l(t))-[f(v)-vf'(v)]dt}\leq 0,$$
            but now $f$ is convex for 
            $$f(u_l(t))-f(v)+vf'(v)\geq f'(v) u_l(t),$$
            so we actually have
            \begin{equation}\label{eq:triangle}
                -\int_0^\epsilon{u(a+\frac{\epsilon}{f'(v)},x)dx}+\int_a^{a+\frac{\epsilon}{f'(v)}}{f'(v)u_l(t)dt}\leq 0.
            \end{equation}
            If we apply Lemma~\ref{lemm:trapeze} to $t_0=a+\frac{\epsilon}{f'(v)}$, $t_1=b$,
            $x_A=0$, $x_B=\epsilon$, $x_C=f'(v)(b-a)-\epsilon$ and $x_D=f'(v)(b-a)$ we get
            \begin{multline*}
                \int_0^\epsilon{u(a+\frac{\epsilon}{f'(v)},x)dx}-\int_{f'(v)(b-a)-\epsilon}^{f'(v)(b-a)}{u(b,x)dx}
                +\int_{a+\frac{\epsilon}{f'(v)}}^b{[f(u(t,(f'(v)(t-a)-\epsilon)^+)}\\
                -f'(v) u(t,(f'(v)(t-a)-\epsilon)^+)]-[f(v)-f'(v)v]dt=0.
            \end{multline*}
            And the convexity of $f$ implies
            $$[f(u(t,(f'(v)(t-a)-\epsilon)^+)-f'(v)
            u(t,(f'(v)(t-a)-\epsilon)^+)]-[f(v)-f'(v)v]\geq 0,$$
            so we actually have
            \begin{equation}\label{eq:trapeze}
                \int_0^\epsilon{u(a+\frac{\epsilon}{f'(v)},x)dx}-\int_{f'(v)(b-a)-\epsilon}^{f'(v)(b-a)}{u(b,x)dx}\leq
                0.
            \end{equation}
            But now adding \eqref{eq:triangle} and \eqref{eq:trapeze} we end up with
            $$\int_{f'(v)(b-a)-\epsilon}^{f'(v)(b-a)}{u(b,x)dx}\geq f'(v)\int_a^{a+\frac{\epsilon}{f'(v)}}{u_l(t)dt}
            ,$$
            and finally dividing by $\epsilon$ and taking $\epsilon\to 0^+$ we obtain the left
            inequality  of \eqref{eq:gauche}.
        \item For the right inequality of \eqref{eq:gauche}. We will proceed in three steps.
            \begin{itemize}
                \item       Step 1, using Lemma~\ref{lemm:trapeze} we get for $c\in (a,b)$
                    \begin{multline*}
                        \int_a^c{f(u(t,\chi(t)))-\dot{\chi}(t)u(t,\chi(t))-f(u(t,\chi(t)+\epsilon))+\dot{\chi}(t)u(t,\chi(t)+\epsilon)
                        dt}\\
                        \int_0^\epsilon{u(a,x)dx}-\int_{\chi(c)}^{\chi(c)+\epsilon}{u(c,x)dx} =0,
                    \end{multline*}
                    using the properties of $\chi$ we get
                    \begin{multline*}
                        \int_a^c{f(v)-f'(v)v-f(u(t,\chi(t)+\epsilon))+f'(v)u(t,\chi(t)+\epsilon)
                        dt}\\
                        \int_0^\epsilon{u(a,x)dx}-\int_{\chi(c)}^{\chi(c)+\epsilon}{u(c,x)dx} =0,
                    \end{multline*}
                    using the convexity of $f$ we get
                    $$\int_0^\epsilon{u(a,x)dx}\geq \int_{\chi(c)}{\chi(c)+\epsilon}{u(c,x)dx},$$
                    dividing by $\epsilon$ and letting $\epsilon \to 0$ we get
                    $$u(a,0^+)\geq u(c,\chi(c)^+)=v.$$
                \item Step 2, since $f'(v)>0$ and $f'$ is increasing we have $f'(u(a,0^+))>0$  so for some
                    point $\bar{x}\in (0,L)$ arbitrarily close to $0$ we have $f'(u(a,x_0))>0$ and
                    considering the minimal backward characteristic $\gamma$  through $(t,\bar{x})$ we
                    have $\dot{\gamma}(t)=f'(\bar{v})>0$ for some $\bar{v}$, therefore if $\bar{x}$ is
                    close enough to $0$ we have a time $c \in (0,a)$ such that $\gamma(c)=0$. If we
                    consider now a time $t\in (c,a)$ should we have $f'(u(t,0^+))<0$ then for $x$ close
                    enough to $0$ we have both 
                    $$f'(u(t,x))<0 \qquad \quad\text{and }\qquad x<\gamma(t),$$
                    but then the maximal backward characteristic through $(t,x)$ will necessarily cross
                    $\gamma$ in $(c,t)$ which is not possible thanks to Theorem~\ref{thm:Daf}. We can thus
                    conclude that for any time $t\in (c,a)$
                    $$f'(u(t,0^+))\geq 0.$$
                    But then using since the boundary condition at $0$ in \eqref{eq:syscar} holds for
                    almost all time $t$ we see that 
                    $$u(t,0^+)=u_l(t),\qquad dt\ a.e.\ \quad\text{in } (c,a).$$
                    And also $f'(u_l(a^-))\geq 0$.
                \item Step 3, let us consider $u_i>u_l(a^-)$. Using Step 2 we can see that $f'(u_i)>0$.
                    Furthermore for a small $\delta>0$ we get
                    $$\forall t\in (a-\delta,a),\qquad f'(u_i)>f'(u_l(t)).$$
                    For $\epsilon>0$, denote by $a_\epsilon$ and $\chi_\epsilon$ the time and curve
                    defined by
                    $$a_\epsilon:=a-\frac{\epsilon}{f'(u_i)},\qquad \forall t\in(a_\epsilon,a),\quad
                    \chi_\epsilon(t):=\epsilon -f'(u_i)(a-t).$$
                    We have $\chi_\epsilon(a_\epsilon)=0$ so using Lemma~\ref{lemm:trapeze} on the
                    triangle of vertices $(a,0)$, $(a,\epsilon)$ and $(a_\epsilon, 0)$ we get
                    $$-\int_0^\epsilon{u(a,x)dx}+\int_{a_\epsilon}^a{f(u(t,0^+))-f(u(t,\chi_\epsilon(t)^-)+f'(u_i)u(t,\chi_\epsilon(t)^-)dt}=0.$$
                    Using the result of the previous step we have then
                    $$-\int_0^\epsilon{u(a,x)dx}+\int_{a_\epsilon}^a{f(u_l(t))-f(u(t,\chi_\epsilon(t)^-)+f'(u_i)u(t,\chi_\epsilon(t)^-)dt}=0.$$
                    but since for $\epsilon$ small enough $a_\epsilon\geq a-\delta$ we have $f'(u_l(t))>0$
                    and $u_l(t)<u_i$. This means that $f(u_l(t))\leq f(u_i)$ so we have
                    $$-\int_0^\epsilon{u(a,x)dx}+\int_{a_\epsilon}^a{f(u_i)-f(u(t,\chi_\epsilon(t)^-)+f'(u_i)u(t,\chi_\epsilon(t)^-)dt}\geq 0.$$
                    But $f$ is convex therefore
                    $$f(u_i)-f(u(t,\chi_\epsilon(t)^-)+f'(u_i)u(t,\chi_\epsilon(t)^-)\leq f'(u_i)u_i,$$
                    and so
                    $$-\int_0^\epsilon{u(a,x)dx}+(a-a_\epsilon)f'(u_i)u_i \geq 0,$$
                    dividing by $\epsilon$ and taking $\epsilon \to 0^+$ we end up with
                    $$-u(a,0^+)+u_i\geq 0,$$
                    so using the result of Step 1 we can conclude
                    $$u_i\geq u(a,0^+)\geq v,$$
                    But $u_i$ was arbitrarily close to $u_l(a^-)$ so as announced
                    $$v\leq u_l(a^-).$$

            \end{itemize}
    \end{itemize}
\end{proof}
\section{Asymptotic Stabilization}
\label{sec:AsymptoticStabilization}
In this section $u$ will be a given solution to the closed loop system \eqref{eq:SystemeFerme}.
We will show estimates \eqref{eq:StabiliteAs}.
\begin{lemm}\label{lem:troisZones}
    Consider $T_1$ given by the following definition 
    $$A_{m,\epsilon}:=\frac{f(u_l(m)-\epsilon)}{2},\qquad
    T_1:=\max\left(\frac{L}{f'(u_l(A_{m,\epsilon}))},\frac{L}{-f'(u_r(A_{m,\epsilon}))}\right).$$
    there exist two Lipschitz functions
    $\beta_1,\beta_2:(T_1,+\infty)\to (0,L)$ such that
    if we consider $(\bt,\bx)\in (T_1,+\infty)\times (0,L)$, then we have the alternatives
    \begin{equation}\label{eq:conditionGauche}
        0<\bx<\beta_1(\bt)\qquad \Rightarrow\qquad  u(\bt,\bx^\pm)\in
        [u_l(m)-\epsilon,u_l(m)+\epsilon],
    \end{equation}
    \begin{equation}\label{eq:conditionCentre}
        \beta_1(\bt)<\bx<\beta_2(\bt)\qquad \Rightarrow \qquad  -\frac{L}{\bt} \leq
        f'(u(\bt,\bx^{\pm}))\leq \frac{L}{\bt}
    \end{equation}
    \begin{equation}\label{eq:conditionDroite}
        \beta_2(\bt)<\bx<L\qquad \Rightarrow \qquad u(\bt,\bx^\pm)=u_r(m).
    \end{equation}
\end{lemm}
\begin{proof}
    We will proceed in mutliple steps.
    \begin{itemize}
        \item We consider $(\bt, \bx)\in (0,+\infty)\times (0,L)$. Using Theorem~\ref{thm:Daf} we get the
            minimal backward characteristics $\gamma$. We call $[a,b]$ its maximal domain of definition.
            Following Theorem~\ref{thm:Daf} and the maximality of $[a,b]$ we see that we have
            $$\left(\gamma(a)=0 \quad\text{and }a>0\right) \quad\text{or }\left(\gamma(a)=L \quad\text{and
            }a>0\right)\quad\text{or } a=0.$$
            \begin{itemize}
                \item In the first case, using
                    Theorem~\ref{thm:resultatBoucleOuverte},
                    we have $u\in \Lip([0,\bt];L^1(0,L))$ therefore the boundary data at $x=0$  
                    $$t\mapsto u_l(m)-\Ac_{\epsilon,\nu}(\Oc_{\alpha,\delta}(u(t,.))),$$
                    is Lipschitz.  Using Proposition~\ref{prop:carBord} we have then
                    $$ u(\bt,\bx^-)=u_l(m)-\Ac_{\epsilon,\nu}(\Oc_{\alpha,\delta}(u(a,.)))\in
                    [u_l(m)-\epsilon,u_l(m)+\epsilon],$$
                    given the definition of $\Ac_{\epsilon,\nu}$.
                \item In the second case, Proposition~\ref{prop:carBord} gives directly
                    $$u(\bt,\bx^-)=u_r(m).$$
                \item Finally in the last case we have $a=0$ and
                    $$\forall t\in [0,b],\qquad \dot{\gamma}(t)=f'(u(\bt,\bx^-)),$$
                    thus
                    $$\gamma(\bt)-\gamma(0)=f'(u(\bt,\bx^-))\bt,$$
                    which means (since $\bx=\gamma(\bt)$)
                    $$f'(u(\bt,\bx))=\frac{\bx-\gamma(0)}{\bt}.$$
                    Using $ 0\leq \gamma(0)\leq L$ we get
                    \begin{equation}\label{eq:limBord}
                        \frac{\bx-L}{\bt}\leq f'(u(\bt,\bx^-))\leq \frac{\bx}{\bt}.
                    \end{equation}
                    which implies
                    $$-\frac{L}{\bt}\leq f'(u(\bt,\bx^-))\leq \frac{L}{\bt}.$$
            \end{itemize}

            Now using Theorem~\ref{thm:Daf} we know that genuine characteristics do not cross.
            Therefore given $\bt$ the set of $\bx$ for which we are in first case, second case or third
            case are connected therefore intervals, they form a partition of $[0,L]$. And from a
            geometrical viewpoint it is obvious that from the left to the right we have points from
            the first case, points from the last case and points from the second case.

        \item At this point we have indeed constructed two functions $\beta_1$ and $\beta_2$ such that
            \eqref{eq:conditionGauche}, \eqref{eq:conditionCentre} and \eqref{eq:conditionDroite}
            hold for $\bx^-$. 

            Since if $0<c<\bx<d<1$ we have 
            $$ u(\bt,\bx^+)=\lim_{\epsilon\to 0^+}u(\bt,(\bx+\epsilon)^-),$$
            \eqref{eq:conditionGauche}, \eqref{eq:conditionCentre}, \eqref{eq:conditionDroite} and
            \eqref{eq:limBord} also hold for $\bx^+$.

            Note that using \eqref{eq:limBord} we get for any $t>0$
            $$\beta_1(t)=0 \qquad \Rightarrow\qquad -\frac{L}{t}\leq f'(u(t,0^+))\leq 0.$$
            We have on one hand
            $$u(t,0^+)\leq 0< u_l(m)-\Ac_{\epsilon,\nu}(\Oc_{\alpha,\delta}(u(t,.))),$$
            and using Remark~\ref{rem:conditionsBord} we can deduce
            $$f(u(t,0^+))\geq f(u_l(m)-\Ac_{\epsilon,\nu}(\Oc_{\alpha,\delta}(u(t,.)))),$$
            which implies
            $$f(u(t,0^+))\geq f(u_l(m)-\epsilon).$$
            On the other hand, if $t\geq T_1$, we have using the definition of $T_1$
            $$f'(u_r(A_{m,\epsilon}))\leq -\frac{L}{T_1}\leq f'(u(t,0^+))\leq 0,$$
            which implies that
            $$u_r(A_{m,\epsilon})\leq u(t,0^+)\leq 0,$$
            and therefore
            $$f(u(t,0^+))\leq f(u_r(A_{m,\epsilon}))=A_{m,\epsilon}=\frac{f(u_l(m)-\epsilon)}{2}<f(u_l(m)-\epsilon).$$
            which is contradictory.  And we can deduce that 
            $$\forall t\geq T_1,\qquad \beta_1(t)>0.$$

            In the same way, using \eqref{eq:limBord} we get for any $t>0$
            $$\beta_2(t)=L\qquad \Rightarrow \qquad 0\leq f'(u(t,L^-))\leq \frac{L}{t}.$$
            On one hand we get
            $$u(t,L^-)\geq 0> u_r(m),$$
            and using Remark~\ref{rem:conditionsBord} we have in particular
            $$f(u(t,L^-))\geq f(u_r(m))=m,$$
            On the other hand, if $t\geq T_1$ we have using the definition of $T_1$
            $$f'(u_l(A_{m,\epsilon}))\geq \frac{L}{T_1}\geq \frac{L}{t}\geq f'(u(t,L^-))\geq 0,$$
            and therefore
            $$u_l(A_{m,\epsilon})\geq u(t,L^-)\geq 0.$$
            We can then obtain
            $$f(u(t,L^-))\leq
            f(u_l(A_{m,\epsilon}))=A_{m,\epsilon}=\frac{f(u_l(m)-\epsilon)}{2}<f(u_l(m)-\epsilon)<f(u_l(m))=m,$$
            which is contradictory. So we can conclude that 
            $$\forall t\geq T_1,\qquad \beta_2(t)<0.$$

        \item It remains to prove that $\beta_1$ and $\beta_2$ are Lipschitz functions. To this end let
            us first remark that those functions are uniquely defined through our previous requirements.

            Now consider $\bt\in (T_1,+\infty)$. Then $\bx:=\beta_1(\bt)\in (0,L)$, so we have a
            unique forward characteristic through $(\bt,\bx)$ let us call it $\gamma_1$,
            defined on a certain interval $[\bt,c]$ with $c>\bt$.  Let us fix $t\in(\bt,c)$.

            If we choose $x\in (0,\gamma_1(t))$ if we consider $\gamma_2$ the minimal backward characteristic through
            $(t,x)$, it is  defined maximally on an interval $[b,t]$. By uniqueness of forward characteristic
            we have
            $$\forall s\in [\max(\bt,b),t],\qquad \gamma_2(s)<\gamma_1(s).$$
            We have two alternatives.
            \begin{itemize}
                \item But then if $b>\bt$ we have $\gamma_2(b)=0$ and $b>0$ therefore $x<\beta_1(t)$.
                \item If on the other hand we have $b\leq \bt$ then $\gamma_2(\bt)<\gamma_1(\bt)=\bx=\beta_1(t)$.
                    But then $\gamma_2$ is also the minimal backward characteristic through
                    $(\bt,\gamma_2(\bt))$ and thus $b>0$  and $\gamma_2(b)=0$ therefore $x<\beta_1(t)$.
            \end{itemize}
            In the end we have proved
            $$\forall x\in (0,\gamma_1(t)),\qquad x<\beta_1(t),$$
            therefore we have $\gamma_1(t)\leq \beta_1(t)$.

            If we choose $x\in (\gamma_1(t),L)$ and consider $\gamma_3$ the minimal backward
            characteristic through $(t,x)$ defined maximally on $[b,t]$. Using the uniqueness of
            forward characteristic we have
            $$\forall x\in [\max(b,\bt),t],\qquad \gamma_3(s)>\gamma_1(s).$$    
            We have two alternatives.
            \begin{itemize}
                \item If $b>\bt$ we have $\gamma_3(b)=L$ and $b>0$ therefore $x>\beta_2(t)\geq
                    \beta_1(t)$.
                \item If on the other hand $b\leq \bt$ we have
                    $\gamma_3(\bt)>\gamma_1(\bt)=\beta_1(\bt)$, but then $\gamma_3$ is also the minimal
                    backward characteristic through $(\bt,\gamma_3(\bt))$ and by construction of
                    $\beta_1$ and $\beta_2$ we have can conclude $x>\beta_1(t)$.
            \end{itemize}
            So we have proved 
            $$\forall x\in (\gamma_1(t),L),\qquad x>\beta_1(t),$$
            which implies $\gamma_1(t)\geq \beta_1(t)$. Since we already had the other inequality
            $\beta_1=\gamma_1$. But this means that $\beta_1$ being a generalized characteristic is
            Lipschitz.

            The same argument works for $\beta_2$.
    \end{itemize}
\end{proof}
\begin{lemm}\label{lem:deuxZones}
    There exists a time $T_2$ independent of the initial data $u_0$ (see \eqref{eq:defT2}
    for an exact formula) and a Lipschitz
    function $\beta:[T_2,+\infty[ \to (0,L)$ satisfying
    $$\forall t\geq T_2,\quad \forall x\in (0,L),\qquad 
    \begin{cases}
        x<\beta(t) \Rightarrow u_l(m)-\epsilon\leq u(t,x^+)\leq u_l(m)+\epsilon\\
        x>\beta(t) \Rightarrow u(t,x^-)=u_r(m)
    \end{cases}
    $$
\end{lemm}
\begin{proof}
    We just need to show the existence of $T_2>0$ independant of $u_0$  such that 
    $$\forall t\geq T_2,\qquad \beta_1(t)=\beta_2(t).$$
    Let us suppose $\beta_1(t)<\beta_2(t)$ for $t\in [T_1,T]$. 
    \begin{itemize}
        \item Using the definition of $\beta_1$ and looking at the minimal backward characteristics
            through $(t,\beta_1(t))$ we get
            $$u_l(m)-\epsilon\leq u(t,\beta_1(t)^-)\leq u_l(m)+\epsilon,$$
            and considering the maximal one
            $$f'(u_r(A_{m,\epsilon}))\leq -\frac{L}{T_1}\leq -\frac{L}{t}\leq f'(u(t,\beta_1(t)^+))\leq
            \frac{L}{t}\leq \frac{L}{T_1}\leq f'(u_l(A_{m,\epsilon})).$$
            But then
            $$u_r(A_{m,\epsilon})\leq u(t,\beta_1(t)^+)\leq u_l(A_{m,\epsilon})<u_l(m)-\epsilon.$$
            Furthermore Theorem~\ref{thm:Daf} grants for almost all $t\in (T_1,T)$
            $$\dot{\beta_1}(t)=\frac{f(u(t,\beta_1(t)^-))-f(u(t,\beta_1(t)^+))}{u(t,\beta_1(t)^-)-u(t,\beta_1(t)^+)}.$$
            Now remark that for $w,z$ the formula
            $$\frac{f(z)-f(w)}{z-w}=\int_0^1{f'(\theta w+(1-\theta)z)d\theta},$$
            show that this function is increasing in both variables therefore
            $$\dot{\beta_1}(t)\geq
            \frac{f(u_l(m)-\epsilon)-f(u_r(A_{m,\epsilon}))}{u_l(m)-\epsilon-u_r(A_{m,\epsilon})}=:c_1>0.$$
        \item Using the definition of $\beta_2$ and looking at the maximal backward characteristics
            through $(t,\beta_2(t))$ we get
            $$u_r(m)=u(t,\beta_2(t)^+),$$
            and considering the minimal one
            $$f'(u_r(A_{m,\epsilon}))\leq -\frac{L}{T_1}\leq -\frac{L}{t}\leq f'(u(t,\beta_2(t)^-))\leq
            \frac{L}{t}\leq \frac{L}{T_1}\leq f'(u_l(A_{m,\epsilon})).$$
            But then
            $$u_r(m)<u_r(A_{m,\epsilon})\leq u(t,\beta_2(t)^-)\leq u_l(A_{m,\epsilon}).$$
            Furthermore Theorem~\ref{thm:Daf} grants for almost all $t\in (T_1,T)$
            $$\dot{\beta_2}(t)=\frac{f(u(t,\beta_2(t)^-))-f(u(t,\beta_2(t)^+))}{u(t,\beta_2(t)^-)-u(t,\beta_2(t)^+)}.$$
            and as before
            $$\dot{\beta_2}(t)\leq
            \frac{f(u_l(A_{m,\epsilon}))-f(u_r(m))}{u_l(A_{m,\epsilon})-u_r(m)}=:-c_2<0.$$
        \item We have then $\beta_1(T_1)\geq 0$ for almost all $t\in (T_1,T)$,
            $\dot{\beta_1}(t)\leq c_1$, so
            $$\beta_1(T)\geq c_1(T-T_1).$$
            In the same way we obtain
            $$\beta_2(t)\leq L-c_2(T-T_1).$$
            But we had supposed $\beta_1(T)\leq \beta_2(T)$ so
            \begin{equation}\label{eq:defT2}
                T\leq T_1+\frac{L}{c_1+c_2}=:T_2.
            \end{equation}
    \end{itemize}
    We have thus shown that
    $$\forall t\geq T_2,\qquad \beta_1(t)=\beta_2(t).$$
\end{proof}
From this result, we get multiple properties.
\begin{rema}\label{rem:propDeuxZones}
    We have the following 
    \begin{gather*}
        \forall t\geq T_2,\quad \forall x\in (0,\beta(t)),\qquad u_l(m)-\epsilon\leq u(t,x)\leq u_l(m)+\epsilon \\
        \intertext{and combined with Definition~\ref{def:defConditionsBords} this implies } 
        u(t,0^+)=u_l(m)-\Ac_{\epsilon,\nu}(\Oc_{\alpha,\delta}(u(t,.))\quad dt\ a.e.\\
        \intertext{We also have }
        \forall t\geq T_2,\quad \forall x\in (\beta(t),L),\qquad u(t,x)=u_r(m),\\
        \intertext{We can then deduce using \eqref{eq:defOc}}
        \forall t\geq T_2,\quad -\frac{u_l(m)-u_r(m)}{2}\leq \Oc_{\alpha,\delta}(u(t,.))\leq
        \frac{u_l(m)-u_r(m)}{2}+\epsilon,\\
        \forall t\geq T_2,\qquad \big(\beta(t)>\alpha+\delta \qquad \Rightarrow \qquad \Oc_{\alpha,\delta}(u(t,.))\geq
        \frac{u_l(m)-u_r(m)}{2}-\epsilon\big),\\
        \forall t\geq T_2,\qquad\big( \beta(t)<\alpha-\delta,\qquad \Rightarrow \qquad
        \Oc_{\alpha,\delta}(u(t,.))=\frac{u_r(m)-u_l(m)}{2}\big),\\
        \intertext{And finally using Theorem~\ref{thm:Daf}}
        \forall t\geq T_2,\qquad \tc \leq   \dot{\beta}(t)\leq \bc. 
    \end{gather*}
    where we have defined
    \begin{equation}\label{eq:vitesseSinguNeg}
        \tc:=\frac{f(u_l(m)-\epsilon)-f(u_r(m))}{u_l(m)-\epsilon-u_r(m)}<0
    \end{equation}
    and
    \begin{equation}\label{eq:vitesseSinguPos}
        \bc:=\frac{f(u_l(m)+\epsilon)-f(u_r(m))}{u_l(m)+\epsilon-u_r(m)}>0.
    \end{equation}
    And note that $\bc$ and $\tc$ tend to $0$ when $\epsilon \to 0$,
    independantly of $\nu,\alpha,\delta$. 
\end{rema}
\begin{lemm}\label{lem:observation}
    Consider $\theta$ given by
    \begin{equation}\label{eq:deftheta}
        \theta:=\max\left(\frac{u_l(m)}{u_l(m)-u_r(m)-\epsilon},\frac{\epsilon-u_r(m)}{u_l(m)+\epsilon-u_r(m)}\right)\in(0,1)
    \end{equation}
    (Note that as $\epsilon$ tends to $0$, $\theta$ tends to a limit strictly less than $1$.)

    Then for $t\geq T_2$,
    $$\beta(t)\geq \alpha+\theta \delta\qquad \Rightarrow \qquad \Oc_{\alpha,\delta}(u(t,.))\geq \frac{u_l(m)-\epsilon}{2},$$
    $$\beta\leq \alpha-\theta \delta \qquad \Rightarrow \qquad \Oc_{\alpha,\delta}(u(t,.))<\frac{u_r(m)}{2}.$$
\end{lemm}
\begin{proof}
    Let us first recall that $u_r(m)<0<u_l(m)-\epsilon$.

    Note that if $\alpha-\delta<\beta(t)<\alpha+\delta$ and we introduce
    $$z:=\frac{\beta(t)-\alpha}{\delta}\in (-1,1),$$
    \begin{itemize}
        \item         we have using the definition of $\Oc_{\alpha,\delta}$ \eqref{eq:defOc}
            \begin{align*}
                \Oc_{\alpha,\delta}(u(t,.))&\geq
                \frac{\beta(t)-(\alpha-\delta)}{2\delta}(u_l(m)-\epsilon)+\frac{\alpha+\delta-\beta(t)}{2\delta}u_r(m)-\frac{u_l(m)+u_r(m)}{2}\\
                &= \frac{\beta(t)-\alpha}{\delta}
                \frac{u_l(m)-u_r(m)-\epsilon}{2}+\frac{\delta(u_l(m)-\epsilon+u_r(m))}{2\delta}-\frac{u_l(m)+u_r(m)}{2}\\
                &= z \frac{u_l(m)-u_r(m)-\epsilon}{2}-\frac{\epsilon}{2}.
            \end{align*}
            but it is clear that this last term is increasing and equals
            $\frac{u_l(m)-\epsilon}{2}$ for  $z$ equal to
            $$\theta_1:=\frac{u_l(m)}{u_l(m)-u_r(m)-\epsilon}\in(0,1).$$
        \item we also have
            \begin{align*}
                \Oc_{\alpha,\delta}(u(t,.))&\leq
                \frac{\beta(t)-(\alpha-\delta)}{\delta}\frac{u_l(m)+\epsilon}{2}+\frac{\alpha+\delta-\beta(t)}{\delta}\frac{u_r(m)}{2}-\frac{u_l(m)+u_r(m)}{2}\\
                &=\frac{\beta(t)-\alpha}{\delta}\frac{u_l(m)-u_r(m)+\epsilon}{2}+\frac{\delta(u_l(m)+\epsilon+u_r(m)}{2\delta}-\frac{u_l(m)+u_r(m)}{2}\\
                &=z\frac{u_l(m)-u_r(m)+\epsilon}{2}+\frac{\epsilon}{2}.
            \end{align*}
            But it is clear that this last term is increasing and equals $\frac{u_r(m)}{2}$
            for $z$ equal to
            $$\theta_2:=-\frac{\epsilon-u_r(m)}{u_l(m)+\epsilon-u_r(m)}\in(-1, 0).$$
        \item So we have
            $$ \beta(t)\geq \alpha+\theta_1 \delta \qquad \Rightarrow
            \Oc_{\alpha,\delta}(u(t,.))\geq \frac{u_l(m)-\epsilon}{2},$$
            $$ \beta(t)\leq \alpha+\theta_2 \delta \qquad \Rightarrow
            \Oc_{\alpha,\delta}(u(t,.))\leq \frac{u_r(m)}{2},$$
            And also by a simple calculation
            $$-1 <\theta_2\leq \theta_1<1.$$
            So taking $\theta:=\max(|\theta_1|,|\theta_2|),$
            we have indeed $\theta\in (0,1)$ and
            $$\alpha-\theta \delta\leq \alpha +\theta_2\delta <\alpha +\theta_1 \delta \leq
            \alpha+\theta \delta.$$
        \item Finally the cases $\beta(t)<\alpha-\delta$ and $\beta(t)>\alpha+\delta$ are obvious
            consequences of the previous calculations.
    \end{itemize}
\end{proof}
\begin{lemm}
    There exists $\epsilon_0$ such that given $\nu_0:=\frac{u_l(m)-u_r(m)}{2},$ for any
    $\nu>\nu_0$ and any $\epsilon \in (0,\min(\epsilon_0,\nu-\frac{u_l(m)-u_r(m)}{2}))$ there exists
    $T_3$ independant of $u_0$ (see \eqref{eq:tempsExplicite} for the exact formula) satisfying
    $$\forall t\geq T_3,\qquad \alpha-\delta<\beta(t)<\alpha+\delta.$$
\end{lemm}
\begin{proof}
    Let us consider $\epsilon_1:=\nu-\frac{u_l(m)-u_r(m)}{2}$. Then using
    \eqref{eq:vitesseSinguNeg}, \eqref{eq:vitesseSinguPos} and \eqref{eq:deftheta},  we know
    that
    $$\frac{\bc}{(1-\theta)f'(u_l(m)-\epsilon)}
    \underset{\epsilon\to0}{\to} 0,$$
    and 
    $$\frac{\tc}{(1-\theta)f'(u_r(m))}
    \underset{\epsilon\to0}{\to} 0.$$
    So there exists $\epsilon_0<\epsilon_1$ such that 
    \begin{equation}\label{eq:comparaisonVitesses}
        \forall \epsilon \in (0,\epsilon_0),\qquad 
        \begin{cases}
            \frac{\bc}{(1-\theta)f'(u_l(m)-\epsilon)}<
            \frac{\delta}{L}\\
            \frac{\tc}{(1-\theta)f'(u_r(m))}< \frac{\delta}{L}.
        \end{cases}
    \end{equation} 
    With such a choice of parameters let us show the result. We will proceed in multiple steps.
    \begin{itemize}
        \item Let us suppose that for an interval $[a,b]\subset [T_2,+\infty[$ we have
            $$\forall t\in [a,b],\qquad \beta(t)\geq \alpha +\theta \delta.$$
            Using Lemma~\ref{lem:observation} we have
            $$\forall t\in [a,b],\qquad \Oc_{\alpha,\delta}(u(t,.))\geq \frac{u_l(m)-\epsilon}{2}>0.$$
            But thanks to $\nu> \nu_0$ and $\epsilon<\epsilon_1$ we deduce
            $$\forall t\in [a,b],\qquad \Ac_{\epsilon,\nu}\left(\Oc_{\alpha,\delta}(u(t,.))\right)\geq
            \frac{\epsilon}{\nu}\frac{u_l(m)-\epsilon}{2}>0.$$
            But then using Remark~\ref{rem:propDeuxZones} we have for almost any $t\in [a,b]$
            \begin{equation}
                0<u_l(m)-\epsilon \leq u(t,0^+)\leq
                u_l(m)-\frac{\epsilon}{\nu}\frac{u_l(m)-\epsilon}{2}<u_l(m).
            \end{equation}
            Now let us suppose that $[a+\frac{L}{f'(u_l(m)-\epsilon)},b]$ non empty and
            consider a time $\bt$ in the interval.  Looking at the minimal backward characteristic through
            $(\bt,\beta(\bt))$ and using Lemmas~\ref{lem:deuxZones} and~\ref{lem:troisZones} we
            have 
            $$u(\bt, \beta(\bt)^-)=u(\bt-\frac{\beta(\bt)}{f'(u(\bt,\beta(\bt)^-))},0^+).$$
            But clearly  using Lemma~\ref{lem:deuxZones} and~\ref{lem:troisZones} we have
            $u(\bt,\beta(\bt)^-)\geq u_l(m)-\epsilon$ so
            $$0\leq \frac{\beta(\bt)}{f'(u(\bt,\beta(\bt)^-))}\leq
            \frac{L}{f'(u_l(m)-\epsilon)},$$
            so we have 
            $$a\leq  \bt-\frac{\beta(\bt)}{f'(u(\bt,\beta(\bt)^-))}\leq b.$$
            From this and Proposition~\ref{prop:carBord} we deduce 
            $$u(\bt,\beta(\bt)^-)\leq u_l(m)-\frac{\epsilon}{\nu}\frac{u_l(m)-\epsilon}{2}.$$
            But looking at the maximal backward characteristic trhough $(bt,\beta(\bt)$ and using Lemmas~\ref{lem:deuxZones} and~\ref{lem:troisZones} 
            we get
            $$u(\bt, \beta(\bt)^+)=u_r(m).$$
            Using Theorem~\ref{thm:Daf} we have shown that if 
            $$b-a\geq \frac{L}{f'(u_l(m)-\epsilon)},$$
            then for almost any time $t$ of the interval $[a+ \frac{L}{f'(u_l(m)-\epsilon)},b]$
            $$\dot{\beta}(t)\leq
            \frac{f(u_l(m)-\frac{\epsilon}{\nu}\frac{u_l(m)-\epsilon}{2})-f(u_r(m))}{u_l(m)-\frac{\epsilon}{\nu}\frac{u_l(m)-\epsilon}{2}-u_r(m)}:=\td<0.$$
            But since $\beta$ is confined inside $(\alpha+\theta \delta,L)$ on $[a,b]$  we require
            $$\td (b-a-\frac{L}{f'(u_l(m)-\epsilon)}+L\geq \alpha+\theta \delta,$$
            which is in fact
            \begin{equation}\label{eq:restrictionTemps1}
                b-a\leq \frac{L}{f'(u_l(m)-\epsilon)}+\frac{\alpha+\theta \delta-L}{\td}.
            \end{equation}
        \item The same method show that if 
            $$\forall t\in[a,b],\qquad \beta(t)\leq \alpha-\theta \delta,$$
            then for almost any $t\in [a,b]$, we have
            $$u(t,0^+)\geq u_l(m)-\frac{\epsilon}{\nu}\frac{u_r(m)}{2}>u_l(m).$$
            Then should we have
            $$ a+\frac{\alpha-\theta \delta}{f'(u_l(m)-\epsilon)} \leq \bt \leq b,$$
            we have
            $$u(\bt,\beta(\bt^-))\geq  u_l(m)-\frac{\epsilon}{\nu}\frac{u_r(m)}{2},$$
            and then
            $$\dot{\beta}(t)\geq
            \bd:=\frac{f(u_l(m)-\frac{\epsilon}{\nu}\frac{u_r(m)}{2})-f(u_r(m))}{u_l(m)-\frac{\epsilon}{\nu}\frac{u_r(m)}{2}-u_r(m)}>0.$$
            And in the end because $\beta$ is supposed to be confined to $[0,\alpha-\theta \delta]$
            for $t\in [a,b]$ we have the restriction
            \begin{equation}\label{eq:restrictionTemps2}
                b-a\leq \frac{\alpha-\theta\delta}{f'(u_l(m)-\epsilon)}+\frac{\alpha-\theta\delta}{\bd},
            \end{equation}
        \item To conclude this part, we have showed that if we define 
            \begin{equation}\label{eq:tempsExplicite}
                T_3=T_2+\max\left(-\frac{L-\alpha-\theta\delta}{\td}-\frac{L}{f'(u_l(m)-\epsilon)}
                ,\frac{\alpha-\theta\delta}{f'(u_l(m)-\epsilon)}+\frac{\alpha-\theta\delta}{\bd}\right).
            \end{equation}
            (see \eqref{eq:restrictionTemps1} and \eqref{eq:restrictionTemps2}) then $\beta$
            cannot be continuously in $(0,\alpha-\theta \delta)$ or $(\alpha+\theta \delta)$ on
            $[T_2,T_3]$. Since $\beta$ is  Lipschitz  we have a time $\bt\in [T_2, T_3]$ such
            that
            $$\beta(\bt)\in (\alpha-\theta \delta,\alpha+\theta \delta).$$
        \item Let us now consider an hypothetical time $b\geq T_3$ such that
            $$\beta(b)\geq \alpha+\delta.$$
            Using the previous result we can define
            $$a:=\sup\{t\in [T_2,b]\ :\ \beta(a)=\alpha+\theta \delta\}.$$
            We have then $\beta(a)=\alpha+\theta \delta,$ and 
            $$\forall t\in [a,b],\qquad \beta(t)\geq \alpha+\theta \delta.$$
            But thanks to Remark~\ref{rem:propDeuxZones} we also know
            $$\forall t\in [a,b],\qquad \dot{\beta}(t)\leq \bc,$$
            therefore 
            $$\alpha+\delta-\alpha-\theta \delta\leq \beta(b)-\beta(a)\leq \bc (b-a).$$
            Therefore
            $$b-a\geq \frac{(1-\theta)\delta}{\bc}.$$
            And thanks to \eqref{eq:comparaisonVitesses} we get
            $$b-a>\frac{L}{f'(u_l(m)-\epsilon)}.$$
            But then for a time $t$ in the (non empty) interval
            $(a+\frac{L}{f'(u_l(m)-\epsilon)},b)$ we have considering the minimal backward
            characteristic
            $$u(t,\beta(t)^-)=u_l(m)-\Ac_{\epsilon,\nu}(\Oc_{\alpha,\delta}(u(s,.))),$$
            for $s$ such that 
            $$\frac{\beta(t)}{t-s}=f'(u(t,\beta(t)^-),$$
            and thus
            $$f'(u_l(m)-\epsilon)\leq \frac{L}{t-s},$$
            but then
            $$s\geq t-\frac{L}{f'(u_l(m)-\epsilon)}\geq b-\frac{L}{f'(u_l(m)-\epsilon)}>a.$$
            Since $\beta(s)\geq \alpha+\theta \delta$ we also have thanks to Lemma~\ref{lem:observation} and $\epsilon<\epsilon_1$
            $$u(t,\beta(t)^-)\leq u_l(m)-\frac{\epsilon}{\nu}\frac{u_l(m)-\epsilon}{2}<0.$$
            Now thanks to Theorem~\ref{thm:Daf} we can conclude that for almost any $t\in
            (a+\frac{L}{f'(u_l(m)-\epsilon)},b)$ we have
            $$\dot{\beta}(t)=\frac{f(u(t,\beta(t)^-))-f(u_r(m))}{u(t,\beta(t)^-)-u_r(m)}<0.$$
            But then
            $$\beta(a+\frac{L}{f'(u_l(m)-\epsilon)}>\beta(b)\geq \alpha+\delta,$$
            and once again
            $$\beta(a+\frac{L}{f'(u_l(m)-\epsilon)}-\beta(a)\leq \bc
            \frac{L}{f'(u_l(m)-\epsilon)},$$
            but we also have
            $$\beta(a+\frac{L}{f'(u_l(m)-\epsilon)}-\beta(a)\geq \alpha+\delta -\alpha-\theta
            \delta=(1-\theta)\delta,$$
            so we end up with the inequality
            $$(1-\theta)\delta\leq \bc \frac{L}{f'(u_l(m)-\epsilon)},$$
            which rewritten
            $$ \frac{\bc}{(1-\theta)f'(u_l(m)-\epsilon)}\geq \frac{\delta}{L}$$
            is incompatible with \eqref{eq:comparaisonVitesses}. 

            In the end we have shown that
            $$\forall b\geq T_3,\qquad \beta(b)< \alpha+\delta.$$
        \item The same method grants
            $$\forall b\geq T_3,\qquad \beta(b)>\alpha-\delta.$$
    \end{itemize}
\end{proof}
\begin{lemm}
    If we call $S$ the function
    $$\forall t>0,\qquad S(t):=\frac{1}{2\delta}\int_{\alpha-\delta}^{\alpha+\delta}{(u(t,x)-\bu_{\alpha,m}(x))dx},$$
    then for $\nu$ sufficiently large (see formula \eqref{eq:choixNu}) one can find a $\Cc^0$ function $\tau:[T_3,+\infty) \to
    [\frac{\alpha-\delta}{f'(u_l(m)+\epsilon)},\frac{\alpha-\delta}{f'(u_l(m)-\epsilon)}]$ such that for
    any time $t\geq T_3$
    $$\dot{S}(t)=\frac{f(u_l(m)-\frac{\epsilon}{\nu}S(t-\tau(t)))-f(u_r(m))}{2\delta}.$$
\end{lemm}
\begin{proof}
    \begin{itemize}
        \item We have seen in Remark~\ref{rem:conditionsBord} that
            \begin{equation}\label{eq:borneS}
                \forall t\geq T_3,\qquad -\frac{u_l(m)-u_r(m)}{2}\leq S(t)\leq
                \frac{u_l(m)-u_r(m)}{2}+\epsilon.
            \end{equation}
            So thanks to our choices of $\epsilon<\nu-\frac{u_l(m)-u_r(m)}{2}$ we have
            $$\Ac_{\epsilon,\nu}(S(t))=\frac{\epsilon}{\nu}S(t).$$
            It is classical that $S$ is Lipschitz and satisfies for almost all $t$
            $$\dot{S}(t)=\frac{f(u(t,\alpha-\delta))-f(u(t,\alpha+\delta))}{2\delta}.$$
            Now using the previous Lemmas we have
            $$\forall t\geq T_3,\qquad u(t,\alpha+\delta)=u_r(m),$$
            and 
            $$u(t,\alpha-\delta)=u(s,0^+),$$
            with
            $$\frac{\alpha-\delta}{t-s}=f'(u(t,\alpha-\delta)).$$
            But we also have thanks to Proposition~\ref{prop:carBord} 
            $$u(s,0^+)=u_l(m)-\Ac_{\epsilon,\nu}(\Oc_{\alpha,\delta}(u(s,.))).$$
            We end up with
            $$\dot{S}(t)=\frac{f(u_l(m)-\Ac_{\epsilon,\nu}(S(t-\tau(t)))-f(u_r(m))}{2\delta}=\frac{f(u_l(m)-\frac{\epsilon}{\nu}S(t-\tau(t))-f(u_r(m))}{2\delta},$$
            with 
            $$\tau(t)=t-s=\frac{\alpha-\delta}{f'(u(t,\alpha-\delta))}.$$
            And we already see that
            $$\frac{\alpha-\delta}{f'(u_l(m)+\epsilon}\leq \tau(t)\leq
            \frac{\alpha-\delta}{f'(u_l(m)-\epsilon)},$$
            thanks to the previous Lemmas. 
        \item All that remains is to prove the regularity of the delay $\tau$. Since at this point it is
            not even clear that $\tau$ is continuous. Thanks to the finite propagation speed, a point
            of discontinuity in time of $\tau$ (thus of $u(t,\alpha-\delta)$) is also a point of
            discontinuity in space. Let us consider $t$ such that
            $$u(t,(\alpha-\delta)^-)>u(t,(\alpha-\delta)^+).$$
            Considering the extremal backward characteristics using Theorem~\ref{thm:Daf},
            Proposition~\ref{prop:carBord} and Remark~\ref{rem:conditionsBord} we get two times
            $t-\frac{\alpha-\delta}{f'(u_l(m)-\epsilon)}\leq t_1<t_2<t$ such that
            $$
            \begin{cases}
                u(t,(\alpha-\delta)^-)=u(t_2,0^+),\quad & \frac{\alpha-\delta}{t-t_2}=f'(u(t_2,0^+))\\
                u(t,(\alpha-\delta)^+)=u(t_1,0^+),\quad & \frac{\alpha-\delta}{t-t_1}=f'(u(t_1,0^+))
            \end{cases}
            $$
            We have therefore 
            $$(t-t_2)f'(u_l(m)-\Ac_{\epsilon,\nu}(S(t_2)))=(t-t_1)f'(u_l(m)-\Ac_{\epsilon,\nu}(S(t_1))).$$
            Now thanks to \eqref{eq:borneS}, \eqref{eq:defAc} and the choices
            $\epsilon<\nu-\frac{u_l(m)-u_r(m)}{2},$ we have in fact
            $$(t-t_2)f'(u_l(m)-\frac{\epsilon}{\nu}S(t_2))=(t-t_1)f'(u_l(m)-\frac{\epsilon}{\nu}S(t_1)).$$
            Now we introduce the function $G$ defined on $[t-
            \frac{\alpha-\delta}{f'(u_l(m)-\epsilon)},t]$ by
            $$G(r):=(t-r)f'(u_l(m)-\frac{\epsilon}{\nu}S(r)).$$
            It is clearly Lipschitz and since $f'$ is $\Cc^1$ we can use the chain rule to get almost
            everywhere
            $$G'(r)=-f'(u_l(m)-\frac{\epsilon}{\nu}S(r))-(t-r)\frac{\epsilon}{\nu}\dot{S}(r)
            f''(u_l(m)-\frac{\epsilon}{\nu}S(r)).$$
            But we have
            $$f'(u_l(m)-\frac{\epsilon}{\nu}S(r))\geq f'(u_l(m)-\epsilon),$$
            and 
            \begin{multline*}
                \left|\dot{S}(r)(t-r)f''(u_l(m)-\frac{\epsilon}{\nu}S(r))\right|\leq
                \frac{\max(f(u_l(m)+\epsilon)-f(u_r(m)), f(u_r(m))-f(u_l(m)-\epsilon))}{2\delta}\\
                \times \frac{\alpha-\delta}{f'(u_l(m)-\epsilon)}\quad\underset{w\in [u_l(m)-\epsilon,u_l(m)+\epsilon]}{\max} f''(w).
            \end{multline*}
            Let us call $M_{m,\epsilon,\delta}$ the righthand side, which is independant of $\nu$ then if 
            \begin{equation}\label{eq:choixNu}
                \frac{\epsilon}{\nu}< \frac{f'(u_l(m)-\epsilon)}{M_{m,\epsilon,\delta}},
            \end{equation}
            we actually have $G'(r)<0$ and then $G(t_1)\neq G(t_2)$ which is contradictory.
            Thus $\tau$ is actually continuous.
    \end{itemize}
\end{proof}
\begin{lemm}
    For $\nu$ sufficiently large (see \eqref{eq:conditionNuFinale}) and $\epsilon$ satisfying the previous conditions, we have constants $C,M_1>0$
    independant of $u_0$ such that
    $$\forall t\geq T_3,\qquad |S(t)|\leq M_1e^{-Ct}\underset{s\in[0,T_3]}{\sup} |S(s)|.$$
\end{lemm}
\begin{proof}
    We just show that we can apply Proposition~\ref{prop:eqRetard} proved in the Appendix for
    $t\in [T_3,+\infty)$. 

    Thanks to the previous Lemmas, we have indeed $S$ is $\Cc^1$ and satisfying
    $$\dot{S}(t)=g(S(t-\tau(t)))$$
    with $\tau$ continuous and
    $$g(z)=\frac{f'(u_l(m)-\frac{\epsilon}{\nu}z)-f(u_r(m))}{2\delta}.$$
    Now thanks to Remark~\ref{rem:conditionsBord} we have
    $$-\frac{u_l(m)-u_r(m)}{2}\leq S(t)\leq \frac{u_l(m)-u_r(m)}{2}+\epsilon.$$
    The delay satisfies
    $$\frac{\alpha-\delta}{f'(u_l(m)+\epsilon)}\leq \tau(t)\leq
    \frac{\alpha-\delta}{f'(u_l(m)-\epsilon)}.$$
    Finally the function satifies $g(0)=0$ and its derivatives is given by
    $$g'(z)=\frac{\epsilon}{2\delta \nu}f''(u_l(m)-\frac{\epsilon}{\nu}z).$$
    So using the uniform convexity of $f$ we get
    $$-\frac{M\epsilon}{2\delta \nu}\leq g'(z)\leq -\frac{m\epsilon}{2\delta \nu},$$
    using
    $$m:=\underset{z\in [u_l(m)-\epsilon,u_l(m)+\epsilon]}{\min} f''(z)>0,\qquad
    M:=\underset{z\in [u_l(m)-\epsilon,u_l(m)+\epsilon]}{\max} f''(z).$$
    We conclude by observing that condition \eqref{eq:hypo2} of the Appendix becomes in our
    case
    \begin{equation}  \label{eq:conditionNuFinale}
        \frac{3(\alpha-\delta)M\epsilon}{2\delta f'(u_l(m)-\epsilon)}<\nu.
    \end{equation}
\end{proof}
\begin{lemm}
    With the previous choices of parameters, we have for $T_4$ given by
    $$T_4=T_3+\frac{L}{f'(u_l(m)-\epsilon)},$$
    two constants $M_2$ and $M_3$ such that
    \begin{equation}\label{eq:stabEtat}
        \forall t\geq T_4,\qquad |\beta(t)-\alpha|\leq M_2e^{-Ct}\underset{s\in[0,T_4]}{\sup} |S(s)|,
    \end{equation}
    and
    \begin{equation}\label{eq:stebPosition}
        \forall t\geq T_4,\quad\forall x<\beta(t),\qquad  |u(t,x)-u_l(m)|\leq M_3e^{-Ct}\underset{s\in[0,T_4]}{\sup} |S(s)|.
    \end{equation}
\end{lemm}
\begin{proof}
    We will proceed in multiple steps.
    \begin{itemize}
        \item Using the previous Lemma and the boundary conditions we have
            $$\forall t\geq T_3,\qquad |u(t,0^+)-u_l(m)|\leq \min\big(\epsilon,\frac{\epsilon
            M_1}{\nu}e^{-Ct}\underset{s\in[0,T_3]}{\sup} |S(s)|\big).$$
            For $t\geq T_3+\frac{L}{f'(u_l(m)-\epsilon)}$ and $x<\beta(t)$ looking at the
            minimal backward characteristic and using Theorem~\ref{thm:Daf} and Proposition~\ref{prop:carBord} we get
            $$|u(t,x)-u_l(m)|\leq M_2 e^{-Ct}\underset{s\in[0,T_3]}{\sup} |S(s)|$$
            with
            $$M_2:=\frac{\epsilon M_1e^{C\frac{\alpha-\delta}{f'(u_l(m)-\epsilon)}}}{\nu}.$$
            And since $T_4>T_3$, \eqref{eq:stabEtat} is now obvious.
        \item Now consider $t\geq T_4$. Let us suppose
            that $\beta(t)\geq \alpha$, we have
            \begin{align*}
                S(t)&\geq \frac{\alpha+\delta-\beta(t)}{2\delta}
                u_r(m)-\frac{u_l(m)+u_r(m)}{2}\\
                    &\qquad +\frac{\beta(t)-\alpha+\delta}{2\delta}\left(u_l(m)-\min\big(\epsilon,
                M_2e^{-Ct}\underset{s\in[0,T_4]}{\sup} |S(s)|\big)\right)\\
                & =
                (\beta(t)-\alpha)\frac{u_l(m)-u_r(m)-\epsilon}{2\delta}-\frac{M_2e^{-Ct}\underset{s\in[0,T_4]}{\sup} |S(s)|}{2}.
            \end{align*}
            And so
            $$0\leq \beta(t)-\alpha \leq M_3e^{-Ct}\underset{s\in[0,T_4]}{\sup} |S(s)|, $$
            with 
            $$K_2:=\frac{\frac{M_2}{2}+M_1}{u_l(m)-u_r(m)-\epsilon}.$$
            The case of $\alpha \geq \beta(t)$ can be treated in the same way.
    \end{itemize}
\end{proof}
\begin{proof}of Theorem~\ref{thm:principal}.

    We just need to write for $t\geq T_4$
    \begin{align*}
        \int_0^L{|u(t,x)-\bu_{\alpha,m}(x)|dx}
        &=\int_0^{\min(\alpha,\beta(t))}{|u(t,x)-\bu_{\alpha,m}(x)|dx}\\
                                               &\qquad +\int_{\min(\alpha,\beta(t))}^{\max(\alpha,\beta(t))}{|u(t,x)-\bu_{\alpha,m}(x)|dx}\\
        &\qquad +\int_{\max(\alpha,\beta(t))}^L{|u(t,x)-\bu_{\alpha,m}(x)|dx}\\
        &\leq \min(\alpha,\beta(t)) M_2e^{-Ct}\underset{s\in[0,T_3]}{\sup} |S(s)|\\
        &\qquad +|\beta(t)-\alpha| 2\max(-u_r(m),u_l(m)+\epsilon)+0\\
        &\leq (LM_2+    2\max(-u_r(m),u_l(m)+\epsilon) M_3) e^{-Ct}\underset{s\in[0,T_4]}{\sup} |S(s)|.
    \end{align*}
    The conclusion then comes from the independance of all the constants from the initial
    data. And   
    $$\underset{s\in[0,T_4]}{\sup} |S(s)|\leq C \int_0^L{|u_0(x)-\bu_{\alpha,m}(x)|dx},$$
    since the semigroup is continuous in $L^1$.
\end{proof}

\appendix
\section{A result on delayed differential equations}
\label{sec:DelayedDifferentialEquation}
\begin{prop}\label{prop:eqRetard}
    Let us consider a function $\theta\in \Cc^1(\R^+)$, a constant $T>0$  and
    a function $g$ such that
    $$\forall t\geq T>0,\qquad \dot{\theta}(t)=g(\theta(t-\tau(t))).$$
    We will suppose the following
    \begin{itemize}
        \item There exists a positive real number $M$ such that
            \begin{equation}\label{eq:borneAPriori}
                \forall t\geq 0,\qquad -M\leq \theta(t)\leq M.
            \end{equation}
        \item We have two positive real numbers $\tau_m$ and $\tau_M$ such that
            \begin{equation}\label{eq:estimDelai}
                \forall t\geq 0,\qquad \tau_m\leq \tau(t)\leq \tau_M.
            \end{equation}
        \item The function $\tau$ is continuous.
        \item We have two positive numbers $c$ and $\epsilon$ such that
            \begin{equation}\label{eq:gLip}
                \forall u\in [-M,M],\qquad -\epsilon\leq g'(u)\leq -c<0.
            \end{equation}
        \item The origin is stationnary
            \begin{equation}\label{eq:statPoint}
                g(0)=0.
            \end{equation}
        \item The following condition holds
            \begin{equation}\label{eq:compatCond1}
                \epsilon(\tau_m+\tau_M)\leq 1.
            \end{equation}
    \end{itemize}
    Then if we define
    $$\forall t\geq T,\qquad B(t):=  \underset{s\in [t-3\tau_M,t]}{\max} |\theta(t)|,$$
    we have the following conclusions.
    \begin{itemize}
        \item If the following condition holds
            \begin{equation}\label{eq:hypo1}
                \epsilon(\tau_m+\tau_M)\leq 1,
            \end{equation}
            then $M$ is non decreasing.
        \item If the following holds
            \begin{equation}\label{eq:hypo2}
                \epsilon(2\tau_M+\tau_m)<1
            \end{equation}
            then $M$ satisfies
            \begin{equation}\label{eq:decroissanceStricte}
                \forall t\geq T,\qquad B(t+3\tau_M)\leq K B(t),
            \end{equation}
            for $K$ given by
            $$K=\frac{1+\epsilon(2\tau_M+\tau_m)c\tau_M}{1+c\tau_M}<1.$$
        \item 
            And from those properties we get
            \begin{equation}\label{eq:expo}
                \forall t\geq t_0\qquad |\theta(t)|\leq e^{\frac{\ln(K)}{3\tau_M}(t-t_0)}B(t_0).
            \end{equation}
    \end{itemize}
\end{prop}
\begin{proof}
    Let us begin by pointing out that using the definition of $B$, properties
    \eqref{eq:statPoint} and \eqref{eq:gLip} of $g$ and properties \eqref{eq:estimDelai} of
    $\tau$ we have that for any time $t$, the function $\theta$ is $\epsilon B(t)$-Lipschitz on
    $[t-2\tau_M,t+\tau_m]$.
    \begin{itemize}
        \item We will now show that $B$ is non increasing. Consider a fixed positive time $t$. 
            We have three alternatives.

            If $\forall s\in [t-\tau_M,t],\qquad \theta(s)>0$, then we have 
            $$\theta(t)>0,\quad \quad\text{and }\quad \forall s\in [t,t+\tau_m],\qquad \theta'(s)<
            0,$$
            but then $\theta$ is decreasing on $[t,t+\tau_m]$ so for $s\in [t,t+\tau_m]$ either 
            $$0\leq \theta(s)\leq \theta(t)\leq B(t),$$
            or $\theta(s)<0$ in which case we have $s_0\in [t,s]$ such that $\theta(s_0)$ but
            then 
            $$|\theta(s)|=|\theta(s)-\theta(s_0)|\leq \epsilon B(t) |s-s_0|\leq \epsilon \tau_m
            B(t)\leq B(t).$$
            In both case we got  (thanks to \eqref{eq:hypo1})
            $$\forall s\in [t,t+\tau_m],\qquad |\theta(s)|\leq B(t),$$
            and therefore 
            $$\forall s\in [t,t+\tau_m],\qquad B(s)\leq B(t).$$

            If $\forall s\in [t-\tau_M,t],\qquad \theta(s)>0,$ the symmetrical argument show
            that 
            $$\forall s\in [t,t+\tau_m],\qquad B(s)\leq B(t).$$

            Finally if we have $s_0\in[t-\tau_M,t]$ such that $\theta(s_0)=0$ then we have
            $$|\theta(s)|=|\theta(s)-\theta(s_0)|\leq \epsilon B(t) |s-s_0|\leq \epsilon
            (\tau_M+\tau_m) B(t)\leq B(t).$$
            and therefore using \eqref{eq:hypo1} 
            $$\forall s\in [t,t+\tau_m],\qquad B(s)\leq B(t).$$
        \item We will now prove \eqref{eq:decroissanceStricte}. We consider a positive time $t$
            which will be fixed. Let us consider $\alpha$ a
            positive number. We will consider once again three alternatives.
            \begin{itemize}
                \item We suppose here that
                    $$\forall s\in [t-2\tau_M,t],\qquad \theta(s)\geq \alpha.$$
                    Using \eqref{eq:gLip} we thus have
                    $$\forall s\in [t-\tau_M,t+\tau_m],\qquad \dot{\theta}(s)\leq -c\alpha,$$
                    but then we can deduce using $\theta(t-\tau_M)\leq B(t)$ that
                    $$\forall s\in [t,t+\tau_m],\qquad \theta(s)\leq B(t)-c\alpha \tau_M.$$
                    We now use the Lipschitz constant of $\theta$ to get
                    $$\forall s\in [t,t+\tau_m],\qquad \theta(s)\geq \theta(t)-\epsilon
                    B(t)(s-t)\geq \alpha -\epsilon B(t) \tau_m,$$
                    Combining the previous estimates we get
                    $$\forall s\in[t,t+\tau_m],\qquad |\theta(s)|\leq \max(\epsilon
                    B(t)\tau_m-\alpha, B(t)-c\alpha \tau_M).$$
                \item We suppose here that
                    $$\forall s\in [t-2\tau_M,t],\qquad \theta(s)\leq -\alpha.$$
                    Using \eqref{eq:gLip} we thus have
                    $$\forall s\in [t-\tau_M,t+\tau_m],\qquad \dot{\theta}(s)\geq c\alpha,$$
                    but then we can deduce using $\theta(t-\tau_M)\geq -B(t)$ that
                    $$\forall s\in [t,t+\tau_m],\qquad \theta(s)\geq -B(t)+c\alpha \tau_M.$$
                    We use the Lipschitz constant for $\theta$ to get 
                    $$\forall s\in [t,t+\tau_m],\qquad \theta(s)\leq \theta(t)+\epsilon
                    B(t)(s-t)\leq -\alpha +\epsilon B(t) \tau_m,$$
                    Combining the previous estimates we get
                    $$\forall s\in[t,t+\tau_m],\qquad |\theta(s)|\leq \max(\epsilon
                    B(t)\tau_m-\alpha, B(t)-c\alpha \tau_M).$$
                \item The last case is now obviously 
                    $$\exists s_0\in [t-2\tau_M,t],\qquad -\alpha \leq \theta(s_0)\leq
                    \alpha.$$
                    But then we have using the Lipschitz constant of $\theta$
                    $$\forall s\in[t,t+\tau_m],\qquad |\theta(s)|\leq |\theta(s_0)|+\epsilon
                    B(t) |s-s_0|\leq \alpha +\epsilon B(t)(2\tau_M+\tau_m).$$
            \end{itemize}
            We can sum up the previous estimates by
            $$\forall s\in [t,t+\tau_m],\qquad |\theta(s)|\leq \max(\alpha +\epsilon
            B(t)(2\tau_M+\tau_m), \epsilon B(t)\tau_m-\alpha, B(t)-c\alpha \tau_M).$$
            But it is clear that 
            $$\forall \alpha\geq 0,\qquad \epsilon B(t)\tau_m-\alpha\leq \alpha +\epsilon
            B(t)(2\tau_M+\tau_m),$$
            thus we have in fact 
            $$\forall s\in [t,t+\tau_m],\qquad |\theta(s)|\leq \max(\alpha +\epsilon
            B(t)(2\tau_M+\tau_m), B(t)-c\alpha \tau_M).$$
            And we can now minimize the righthandside with respect to $\alpha$. Since the
            functions are affine (one increasing the other decreasing) the corresponding $\alpha$ satisfies
            $$\alpha +\epsilon
            B(t)(2\tau_M+\tau_m)= B(t)-c\alpha \tau_M,$$
            which is 
            $$\alpha=\frac{1-\epsilon(2\tau_M+\tau_m)}{1+c\tau_m}.$$ 
            and so we end up with
            $$\forall s\in [t,t+\tau_m],\qquad
            |\theta(s)|\leq\left(\frac{1+c\tau_M\epsilon(2\tau_M+\tau_m)}{1+c\tau_M}\right)B(t).$$
            Finally by bootstrapping the result and using the fact that $B$ is non increasing
            we have
            $$\forall s\in [t,t+3\tau_M],\qquad
            |\theta(s)|\leq\left(\frac{1+c\tau_M\epsilon(2\tau_M+\tau_m)}{1+c\tau_M}\right)B(t),$$
            which is as announced
            $$\forall t\geq T,\qquad B(t+3\tau_M)\leq K B(t).$$
        \item To get \eqref{eq:expo} we consider $t>t_0$ and denote $ N$ the integer satisfying
            $$t-3(N+1)\tau_M\leq t_0\leq t-3N\tau_M \Leftrightarrow N\leq \frac{t-t_0}{3\tau_M}\leq N+1.$$
            We then have
            \begin{align*}  
                B(t)&\leq KB(t-3\tau_M)\\
                    &\leq K^2B(t-2(3\tau_M))\\
                    &\leq K^NB(t-N(3\tau_M))\\
                    &\leq e^{\ln(K) N}B(t_0)\\
                    &\leq e^{\frac{\ln(K)}{3\tau_M}(t-t_0)} B(t_0).
            \end{align*}
    \end{itemize}
\end{proof}
\bibliographystyle{plain}
\bibliography{references}

\begin{thebibliography}{10}

\bibitem{AGV2014}
Adimurthi, Shyam~Sundar Ghoshal, and G.~D. Veerappa~Gowda.
\newblock Exact controllability of scalar conservation laws with strict convex
  flux.
\newblock {\em Math. Control Relat. Fields}, 4(4):401--449, 2014.

\bibitem{AWC2006}
Kaouther Ammar, Petra Wittbold, and Jose Carrillo.
\newblock Scalar conservation laws with general boundary condition and
  continuous flux function.
\newblock {\em J. Differential Equations}, 228(1):111--139, 2006.

\bibitem{AC2005}
Fabio Ancona and Giuseppe~Maria Coclite.
\newblock On the attainable set for {T}emple class systems with boundary
  controls.
\newblock {\em SIAM J. Control Optim.}, 43(6):2166--2190, 2005.

\bibitem{AM1998}
Fabio Ancona and Andrea Marson.
\newblock On the attainable set for scalar nonlinear conservation laws with
  boundary control.
\newblock {\em SIAM J. Control Optim.}, 36(1):290--312, 1998.

\bibitem{AM2007}
Fabio Ancona and Andrea Marson.
\newblock Asymptotic stabilization of systems of conservation laws by controls
  acting at a single boundary point.
\newblock In {\em Control methods in {PDE}-dynamical systems}, volume 426 of
  {\em Contemp. Math.}, pages 1--43. Amer. Math. Soc., Providence, RI, 2007.

\bibitem{ADGR2015}
Boris Andreianov, Carlotta Donadello, Shyam~Sundar Ghoshal, and Ulrich
  Razafison.
\newblock On the attainable set for a class of triangular systems of
  conservation laws.
\newblock {\em J. Evol. Equ.}, 15(3):503--532, 2015.

\bibitem{ADM2017}
Boris~P. Andreianov, Carlotta Donadello, and Andrea Marson.
\newblock On the attainable set for a scalar nonconvex conservation law.
\newblock {\em SIAM J. Control Optim.}, 55(4):2235--2270, 2017.

\bibitem{BLN1979}
C.~Bardos, A.~Y. le~Roux, and J.-C. N\'ed\'elec.
\newblock First order quasilinear equations with boundary conditions.
\newblock {\em Comm. Partial Differential Equations}, 4(9):1017--1034, 1979.

\bibitem{BC2016}
Georges Bastin and Jean-Michel Coron.
\newblock {\em Stability and boundary stabilization of 1-{D} hyperbolic
  systems}, volume~88 of {\em Progress in Nonlinear Differential Equations and
  their Applications}.
\newblock Birkh\"auser/Springer, [Cham], 2016.
\newblock Subseries in Control.

\bibitem{BLDPB2017}
S\'ebastien Blandin, Xavier Litrico, Maria~Laura Delle~Monache, Benedetto
  Piccoli, and Alexandre Bayen.
\newblock Regularity and {L}yapunov stabilization of weak entropy solutions to
  scalar conservation laws.
\newblock {\em IEEE Trans. Automat. Control}, 62(4):1620--1635, 2017.

\bibitem{BC2002}
Alberto Bressan and Giuseppe~Maria Coclite.
\newblock On the boundary control of systems of conservation laws.
\newblock {\em SIAM J. Control Optim.}, 41(2):607--622, 2002.

\bibitem{CR2015}
Rinaldo~M. Colombo and Elena Rossi.
\newblock Rigorous estimates on balance laws in bounded domains.
\newblock {\em Acta Math. Sci. Ser. B Engl. Ed.}, 35(4):906--944, 2015.

\bibitem{CEGGP2017}
Jean-Michel Coron, Sylvain Ervedoza, Shyam~Sundar Ghoshal, Olivier Glass, and
  Vincent Perrollaz.
\newblock Dissipative boundary conditions for {$2\times 2$} hyperbolic systems
  of conservation laws for entropy solutions in {BV}.
\newblock {\em J. Differential Equations}, 262(1):1--30, 2017.

\bibitem{Dafermos1977}
C.~M. Dafermos.
\newblock Generalized characteristics and the structure of solutions of
  hyperbolic conservation laws.
\newblock {\em Indiana Univ. Math. J.}, 26(6):1097--1119, 1977.

\bibitem{Dafermos2016}
Constantine~M. Dafermos.
\newblock {\em Hyperbolic conservation laws in continuum physics}, volume 325
  of {\em Grundlehren der Mathematischen Wissenschaften [Fundamental Principles
  of Mathematical Sciences]}.
\newblock Springer-Verlag, Berlin, fourth edition, 2016.

\bibitem{Filippov1960}
A.~F. Filippov.
\newblock Differential equations with discontinuous right-hand side.
\newblock {\em Mat. Sb. (N.S.)}, 51 (93):99--128, 1960.

\bibitem{GG2007}
O.~Glass and S.~Guerrero.
\newblock On the uniform controllability of the {B}urgers equation.
\newblock {\em SIAM J. Control Optim.}, 46(4):1211--1238, 2007.

\bibitem{Glass2007}
Olivier Glass.
\newblock On the controllability of the 1-{D} isentropic {E}uler equation.
\newblock {\em J. Eur. Math. Soc. (JEMS)}, 9(3):427--486, 2007.

\bibitem{Glass2014}
Olivier Glass.
\newblock On the controllability of the non-isentropic 1-d euler equation.
\newblock {\em J. Differential Equations}, 257(3):638--719, 2014.

\bibitem{GZ2017}
Laurent Gosse and Enrique Zuazua.
\newblock Filtered gradient algorithms for inverse design problems of
  one-dimensional {B}urgers equation.
\newblock In {\em Innovative algorithms and analysis}, volume~16 of {\em
  Springer INdAM Ser.}, pages 197--227. Springer, Cham, 2017.

\bibitem{Horsin1998}
T.~Horsin.
\newblock On the controllability of the {B}urgers equation.
\newblock {\em ESAIM Control Optim. Calc. Var.}, 3:83--95, 1998.

\bibitem{Kruzkov1970}
S.~N. Kru\v~zkov.
\newblock First order quasilinear equations with several independent variables.
\newblock {\em Mat. Sb. (N.S.)}, 81 (123):228--255, 1970.

\bibitem{leRoux1976}
Alain~Yves le~Roux.
\newblock On the convergence of the {G}odounov's scheme for first order quasi
  linear equations.
\newblock {\em Proc. Japan Acad.}, 52(9):488--491, 1976.

\bibitem{leRoux1977}
Alain~Yves le~Roux.
\newblock \'etude du probl\`eme mixte pour une \'equation quasi-lin\'eaire du
  premier ordre.
\newblock {\em C. R. Acad. Sci. Paris S\'er. A-B}, 285(5):A351--A354, 1977.

\bibitem{Leautaud2012}
Matthieu L\'eautaud.
\newblock Uniform controllability of scalar conservation laws in the vanishing
  viscosity limit.
\newblock {\em SIAM J. Control Optim.}, 50(3):1661--1699, 2012.

\bibitem{LY2017}
Tatsien Li and Lei Yu.
\newblock One-sided exact boundary null controllability of entropy solutions to
  a class of hyperbolic systems of conservation laws.
\newblock {\em J. Math. Pures Appl. (9)}, 107(1):1--40, 2017.

\bibitem{MNR1996}
J.~M\'alek, J.~Nevcas, M.~Rokyta, and M.~Rocircuvzivcka.
\newblock {\em Weak and measure-valued solutions to evolutionary {PDE}s},
  volume~13 of {\em Applied Mathematics and Mathematical Computation}.
\newblock Chapman \& Hall, London, 1996.

\bibitem{MT1999}
Corrado Mascia and Andrea Terracina.
\newblock Large-time behavior for conservation laws with source in a bounded
  domain.
\newblock {\em J. Differential Equations}, 159(2):485--514, 1999.

\bibitem{Oleinik1956}
O.~A. Ole\u\i~nik.
\newblock On discontinuous solutions of non-linear differential equations.
\newblock {\em Dokl. Akad. Nauk SSSR (N.S.)}, 109:1098--1101, 1956.

\bibitem{Otto1996}
Felix Otto.
\newblock Initial-boundary value problem for a scalar conservation law.
\newblock {\em C. R. Acad. Sci. Paris S\'er. I Math.}, 322(8):729--734, 1996.

\bibitem{Perrollaz2012}
Vincent Perrollaz.
\newblock Exact controllability of scalar conservation laws with an additional
  control in the context of entropy solutions.
\newblock {\em SIAM J. Control Optim.}, 50(4):2025--2045, 2012.

\bibitem{Perrollaz2013}
Vincent Perrollaz.
\newblock Asymptotic stabilization of entropy solutions to scalar conservation
  laws through a stationary feedback law.
\newblock {\em Ann. Inst. H. Poincar\'e Anal. Non Lin\'eaire}, 30(5):879--915,
  2013.

\bibitem{Rudin1973}
Walter Rudin.
\newblock {\em Functional analysis}.
\newblock McGraw-Hill Book Co., New York-D\"usseldorf-Johannesburg, 1973.
\newblock McGraw-Hill Series in Higher Mathematics.

\end{thebibliography}

\end{document}